\documentclass[a4paper]{article}

\usepackage[english]{babel}
\usepackage[T1]{fontenc}
\usepackage{amssymb}
\usepackage{amsthm}
\usepackage{makeidx}
\usepackage{color}
\usepackage{fullpage}
\usepackage{tikz}
\usepackage{mathtools}
\usepackage{multirow}
\usepackage[all]{xy}
\usepackage{comment}
\usepackage{authblk}
\usepackage[backend=biber,style=numeric,natbib=true]{biblatex}

\usepackage[shortlabels]{enumitem}
\usepackage{csquotes}

\makeatletter
\g@addto@macro\bfseries{\boldmath}
\makeatother

\usepackage[a4paper,top=3cm,bottom=2cm,left=2cm,right=2cm,marginparwidth=1.75cm]{geometry}

\usepackage{amsmath}
\usepackage{graphicx}
\usepackage[colorlinks=true, allcolors=blue]{hyperref}
\usepackage[all]{xy}

\def \N {{\mathbb N}}
\def \C {{\mathbb C}}
\def \Z {{\mathbb Z}}
\def \Q {{\mathbb Q}}
\def \R {{\mathbb R}}
\def \O {{\mathcal O}}
\def \P {{\mathcal P}}

\providecommand{\keywords}[1]{\textbf{\textit{Keywords---}} #1}

\providecommand{\MSC}[1]{\textbf{\textit{MSC---}} #1}

\newtheorem{defi}{Definition}[section]

\newtheorem{teo}[defi]{Theorem}
\newtheorem{coro}[defi]{Corollary}
\newtheorem{lema}[defi]{Lemma}

\newtheorem{pro}[defi]{Proposition}
\newtheorem{rmk}[defi]{Remark}

\title{Additive structure of non-monogenic simplest cubic fields}
\author[1,2]{Daniel Gil-Muñoz}
\author[3,4]{Magdaléna Tinková}
\affil[1]{Department of Algebra, Faculty of Mathematics and Physics, Charles University,
Sokolovsk\'{a} 83, 186 00 Praha 8, Czech Republic\bigskip}
\affil[2]{Dipartimento di Matematica, Università di Pisa, Largo B. Pontecorvo, 5, 56127 Pisa, Italy\bigskip}
\affil[3]{Faculty of Information Technology, Czech Technical University in Prague, Th\'akurova 9, 160 00 Praha 6, Czech Republic\bigskip}
\affil[4]{Institute of Analysis and Number Theory, TU Graz, Kopernikusgasse 24/II, 8010 Graz, Austria}

\date{}

\begin{document}
\maketitle

\begin{abstract}
We consider Shanks' simplest cubic fields $K$ for which the index $[\mathcal{O}_K:\mathbb{Z}[\rho]]$ of a root $\rho$ of the defining parametric polynomial is $3$. For them, we study the additive indecomposables of $K$ and provide a complete list of them. Moreover, we use the knowledge of the indecomposables to prove some interesting consequences on the arithmetic of $K$. Mainly, we obtain good bounds on the ranks of universal quadratic forms over $K$ and prove that the Pythagoras number of $\mathcal{O}_K$ is $6$.
\end{abstract}

\MSC{11R16, 11R80, 11R04.}


\section{Introduction}

An element of a totally real number field $K$ is totally positive if it is mapped to a positive real number by each of the embeddings of $K$. This notion is an analogue of positive rational numbers, and it is well suited to describe the additive structure of the ring of algebraic integers $\mathcal{O}_K$. Namely, an indecomposable of $K$ is a totally positive element of $\mathcal{O}_K$ that cannot be written as a sum of other two totally positive elements of $\mathcal{O}_K$. With this definition, it is clear that every totally positive algebraic integer in $K$ is a sum of indecomposables of $K$.

The study of indecomposable elements has been proved as a useful tool to study a number of arithmetic properties of $K$. Blomer and Kala \cite{blomerkala2015,kala2016} established a connection of these elements with the minimal rank $m(K)$ of universal quadratic forms over $K$, which are those quadratic forms with coefficients in $\mathcal{O}_K$ (or an $\mathcal{O}_K$-order) that represent all totally positive elements of $\mathcal{O}_K$. For a totally real quadratic field $K=\mathbb{Q}(\sqrt{D})$, Perron \cite{perron}, Dress and Scharlau \cite{dressscharlau} found a complete list of indecomposable integers in $K$ in terms of the continued fraction expansion of $\frac{-1+\sqrt{D}}{2}$ if $D\equiv1\,(\mathrm{mod}\,4)$ and of $\sqrt{D}$ otherwise. Blomer and Kala used this description to show that for every positive integer $n$ there are infinitely many totally real quadratic fields $K$ such that $m(K)\geq n$. In other words, there are infinitely many real quadratic fields without universal quadratic forms of rank $n$ arbitrarily large. In the following years analogous results were obtained for other families of number fields \cite{kalasvodoba,yatsyna,kala2022,blomerkala2018}.

This substantial progress on quadratic fields motivated the search of indecomposables in families of totally real number fields with higher degree. Among these, the family of simplest cubic fields, originally studied by Shanks \cite{shanks1974}, presents a nice arithmetic behaviour. These are the totally real cubic fields generated by a polynomial of the form $$f(x)=x^3-ax^2-(a+3)x-1,\quad a\in\mathbb{Z}_{\geq -1},$$ meaning that $K=\mathbb{Q}(\rho)$ for a root $\rho$ of $f$. 

Kala and the second author \cite{kalatinkova} provided a complete list of the indecomposables in $\mathbb{Z}[\rho]$, which led to some bounds on the minimal rank of universal quadratic forms (see \cite[Theorem 1.1]{kalatinkova}). To do so, they proved that all of the indecomposables in $\mathbb{Z}[\rho]$ are, up to multiplication by totally positive units, contained in two specific parallelepipeds at the Minkowski space $\mathbb{R}^3$, whose nodes depend on a proper pair of totally positive units (according to the terminology in \cite{thomasvasquez}). These considerations led to a geometric method consisting in finding all lattice points within those two parallelepipeds and identifying the indecomposables among them. In \cite[Section 8]{kalatinkova} and \cite{tinkova}, full lists of indecomposables of an analogous order of fields in some other families are given. Moreover, inspired by the connection of indecomposables with the continued fraction in real quadratic fields, the authors in \cite{kst} started to study the question whether indecomposables in fields of higher degrees can be obtained by some of the multidimensional continued fraction algorithms. For some partial results on indecomposable integers in real biquadratic fields, see also \cite{cechetal,krasenskytinkovazemkova}.

If $K=\mathbb{Q}(\rho)$ is a simplest cubic field, it is known that $\mathrm{disc}(\rho)=\Delta^2$, where $\Delta=a^2+3a+9$. Moreover, for a positive density of $a$, this is just the discriminant of $K$, and those are the cases in which $\mathcal{O}_K=\mathbb{Z}[\rho]$. In this situation, the above-mentioned classification by Kala and the second author provides all the indecomposables in the ring of integers of the simplest cubic field.

In general, a simplest cubic field $K$ need not be monogenic. Kashio and Sekigawa \cite{kashiosekigawa} gave a characterization for this condition. Namely, they proved that $K$ is monogenic if and only if the module index $\delta\coloneqq[\mathcal{O}_K:\mathbb{Z}[\rho]]$ is the cube of some integer number or $K$ is defined by more than one value of the parameter $a$. In this paper, we shall specialize to the case when $\delta=3$. We will see that in that case, our field has integral basis $$g_1=1,\quad g_2=\rho,\quad g_3=\frac{1+\rho+\rho^2}{3}.$$ This family had already been considered in a different (but equivalent) form in \cite[Theorem 3]{canovasorvay}. We will determine the lattice points at each of the two parallelepipeds containing indecomposables of $K$.

The main result in this paper is a complete list of indecomposables for fields in this family, which we prove in Section \ref{sec:indec}.

\begin{teo}\label{teolistindec} Let $K=\mathbb{Q}(\rho)$ be a simplest cubic field with $[\mathcal{O}_K:\mathbb{Z}[\rho]]=3$. 
Up to multiplication by totally positive units, the indecomposable integers in $K$ are 1,
\begin{itemize}
    \item[(i)] $g_3$,
    \item[(ii)] $-g_1-(r+1)g_2+3g_3=-r\rho+\rho^2$ where $1\leq r\leq\frac{a}{3}$,
    \item[(iii)] $-(2v+1)g_1-(v(a+3)+2)g_2+3(v+1)g_3=-v-(v(a+2)+1)\rho+(v+1)\rho^2$ where $\frac{2a}{3}+1\leq v\leq a$,
    \item[(iv)] $-(2v+1)g_1-(v+1)(a+2)g_2+3(v+1)g_3=-v-(v(a+2)+a-v+1)\rho+(v+1)\rho^2$ where $0\leq v\leq \frac{a}{3}-1$,
    \item[(v)] $-(2v+1)g_1-(v(a+3)+r+1)g_2+(3v+2)g_3=\\-v-(v(a+2)+r)\rho+(v+1)\rho^2-g_3$ where $ 0\leq v\leq\frac{a}{3}-1$ and $\frac{a}{3}+1\leq r\leq\frac{2a}{3}-v$,
    \item[(vi)] $-(r+1)g_2+g_3=1-r\rho+\rho^2-2g_3$ where $0\leq r\leq\frac{a}{3}-1$,
    \item[(vii)] $-(2v+2)g_1-(v(a+3)+\frac{2a}{3}+3)g_2+(3v+4)g_3=\\-v-(v(a+2)+\frac{2a}{3}+1)\rho+(v+2)\rho^2-2g_3$ where $0\leq v\leq\frac{a}{3}-1$,
    \item[(viii)] $-(2v+2)g_1-(v(a+3)+\frac{4a}{3}-v+3)g_2+(3v+4)g_3=\\-v-(v(a+2)+\frac{4a}{3}-v+1)\rho+(v+2)\rho^2-2g_3$ where $\frac{a}{3}\leq v\leq\frac{2a}{3}-1$.
\end{itemize}
The points from (ii) to (iv) have minimal trace $2$, while the remaining ones have minimal trace $1$. The total number of indecomposables in $K$ (up to multiplication by totally positive units) is $\frac{a^2+3a}{18}+2a+2$.
\end{teo}

In this statement, the minimal trace of an algebraic integer $\alpha\in\mathcal{O}_K$ is the minimum among all numbers of the form $\mathrm{Tr}(\delta\alpha)$, for totally positive elements $\delta$ in the codifferent $$\mathcal{O}_K^{\vee}=\{\delta\in K\,|\,\mathrm{Tr}(\delta\alpha)\in\mathbb{Z}\text{ for all }\alpha\in\O_K\}$$ of $K$. Its importance in the search of indecomposables relies in the fact that every totally positive algebraic integer of minimal trace $1$ is indecomposable. On the other hand, in \cite[Theorem 1.1]{tinkova}, the second author proved that the minimal trace of indecomposables in totally real cubic orders can be arbitrarily large.

Theorem \ref{teolistindec} is the first result providing the whole structure of indecomposables for a family of non-monogenic totally real number fields, although some results were achieved after the first submission of this paper, see, e.g., \cite{man}. It is also remarkable that this classification is substantially different from the one obtained by Kala and the second author for the indecomposables of $\mathbb{Z}[\rho]$. For instance, many indecomposables of $\mathbb{Z}[\rho]$ are no longer indecomposables of $\mathcal{O}_K$ in this case.

In Section \ref{sect:consequences}, we shall study some interesting consequences of Theorem \ref{teolistindec} on the arithmetic of this subfamily of the simplest cubic fields.

Lemmermeyer and Peth\"o \cite{lemmermeyerpetho} proved that for every $\gamma\in\mathbb{Z}[\rho]$ not associated with a rational integer, the norm $N(\gamma)$ of $\gamma$ satisfies $|N(\gamma)|\geq 2a+3$. In Proposition \ref{prosmallestnorm} we will see that except for some small values of $a$, this is also true for our non-monogenic subfamily. On the other hand, we will also study the largest possible norm of the indecomposables, and concretely we will prove that $|N(\alpha)|\leq\frac{(a^2+3a+9)^2}{729}$, again with some exceptions (see Proposition \ref{pro:largestnorm}).

The structure of indecomposables of a totally real number field $K$ can be used to find the Pythagoras number of its ring of integers $\mathcal{O}_K$ or any order in it. The Pythagoras number of a ring is the minimal number $s\in\mathbb{Z}$ such that every sum of a finite number of squares can be written as a sum of $s$ squares. The second author \cite{tinkovapyth} showed that in the case of a simplest cubic field $K=\mathbb{Q}(\rho)$, the Pythagoras number of $\mathbb{Z}[\rho]$ is exactly $6$ if $a\geq 3$. In the case that $\delta=3$, we will show that the Pythagoras number of $\mathcal{O}_K$ is again $6$. 

We shall use the knowledge of the indecomposables for this subfamily of the simplest cubic fields $K$ to find estimates for the minimal rank $m(K)$ of universal quadratic forms over $K$, in the style of the stated ones in \cite[Theorem 1.1]{kalatinkova}.

\section{Preliminaries}\label{sect:prelim}

Let $K$ be a totally real number field of degree $d$ and let $\sigma_1,\dots,\sigma_d\colon K\hookrightarrow\mathbb{R}$ be its embeddings. By trace $\text{Tr}(\alpha)$ and norm $N(\alpha)$ of $\alpha\in K$, we will mean
\[
\text{Tr}(\alpha)=\sum_{i=1}^d\sigma_i(\alpha) \hspace{1cm}\text{ and }\hspace{1cm} N(\alpha)=\prod_{i=1}^d\sigma_i(\alpha).
\]
An element $\alpha\in K$ is totally positive if $\sigma_i(\alpha)>0$ for every $1\leq i\leq d$, and we will denote it by $\alpha\succ 0$. The subset of the totally positive integers in $K$ is denoted by $\mathcal{O}_K^+$. Moreover, by $\alpha\succ\beta$, we will mean $\alpha-\beta\succ 0$. We will also use symbol $\succeq$ to include equality between our elements.

A totally positive integer $\alpha\in\mathcal{O}_K$ is said to be indecomposable if it cannot be written as the sum of other two totally positive integers in $\mathcal{O}_K$. The codifferent of $\mathcal{O}_K$ is $$\mathcal{O}_K^{\vee}=\{\delta\in K\,|\,\text{Tr}(\delta\alpha)\in\mathbb{Z}\hbox{ for every }\alpha\in\mathcal{O}_K\}.$$ The subset of the totally positive integers in $\mathcal{O}_K^{\vee}$ is denoted by $\mathcal{O}_K^{\vee,+}$. For $\alpha\in\mathcal{O}_K^{+}$, the number $$\mathrm{min}_{\delta\in\mathcal{O}_K^{\vee,+}}\mathrm{Tr}(\delta\alpha)$$ will be referred to as the minimal trace of $\alpha$. It is immediate that the totally positive algebraic integers with minimal trace $1$ are indecomposable.

We proceed to summarize the method introduced in \cite{kalatinkova} to find the indecomposables in $K$. Let us embed $K$ on the Minkowski space $\mathbb{R}^d$ by means of an integral basis of $K$.
Following \cite[Section 1, p. 2]{thomasvasquez}, we introduce the following terminology:

\begin{defi}\label{def:properpair} A  pair $(\varepsilon_1,\varepsilon_2)$ of units in $K$ is proper if the following statements hold:
\begin{itemize}
    \item[(1)] The subgroup of $K^*$ generated by $\{\varepsilon_1,\varepsilon_2\}$ has rank $2$.
    \item[(2)] The set $\{1,\varepsilon_1,\varepsilon_2\}$ is a $\mathbb{Q}$-basis of $K$.
    \item[(3)] In the linear combination of $\varepsilon_1\varepsilon_2$ with respect to the above basis, the coefficient before $1$ is negative.
\end{itemize}
\end{defi}

On the other hand, the parallelepiped with nodes $\ell_1,\dots,\ell_e\in\mathbb{Z}^d$ is $$\mathcal{D}(\ell_1,\dots,\ell_e)=[0,1]\ell_1+\dots+[0,1]\ell_e.$$ The key result is the following refinement of \cite[Theorem 1]{thomasvasquez} (see \cite[Section 4, p. 11-12]{kalatinkova}):

\begin{pro}\label{pro:ppidsmonogscf} Let $K$ be a totally real cubic field and let $\{\varepsilon_1,\varepsilon_2\}$ be a proper pair of totally positive units of $K$. The indecomposables of $K$ lie, up to multiplication by totally positive units in $K$, in either of the parallelepipeds $$\mathcal{D}(1,\varepsilon_1,\varepsilon_2),\quad\mathcal{D}(1,\varepsilon_1,\varepsilon_1\varepsilon_2^{-1}).$$
\end{pro}

This method was applied successfully to the case that $K$ is a simplest cubic field (see Section \ref{sectscf} below) with $\mathcal{O}_K=\mathbb{Z}[\rho]$, leading to a list of lattice points among which the indecomposables were identified.

\begin{teo}\cite[Theorem 1.2]{kalatinkova}
Let $\rho$ be a root of the polynomial $x^3-ax^2-(a+3)x-1$ where $a\in\Z_{\geq -1}$.
Up to multiplication by totally positive units, the indecomposables in the order $\mathbb{Z}[\rho]$ of $\Q(\rho)$ are $1$,
\begin{itemize}
    \item the element $1+\rho+\rho^2$ with minimal trace $2$,
    \item the elements in the set $$\blacktriangle=\{-v-w\rho+(v+1)\rho^2\,|\,0\leq v\leq a,\,v(a+2)+1\leq w\leq (v+1)(a+1)\},$$ all of which have minimal trace $1$.
\end{itemize}
\end{teo}

\subsection{Simplest cubic fields and their monogenity}\label{sectscf}

A simplest cubic field is a number field $K=\mathbb{Q}(\rho)$, where $\rho$ is a root of the polynomial $$f(x)=x^3-ax^2-(a+3)x-1,\quad a\in\mathbb{Z}_{\geq -1}.$$ It is a cyclic cubic extension of $\mathbb{Q}$ (and then totally real).

Let $\rho'$ and $\rho''$ be the conjugates of $\rho$, that is, the other roots of $f$. When $a\geq7$, reordering these roots if necessary, we can assume that \cite{lemmermeyerpetho} $$a+1<\rho<a+1+\frac{2}{a},\quad-1-\frac{1}{a}<\rho'<-1-\frac{1}{2a},\quad-\frac{1}{a+2}<\rho''<-\frac{1}{a+3}.$$ We also have the explicit expressions $$\rho'=-1-\frac{1}{\rho}=a+2+a\rho-\rho^2\quad \text{and} \quad \rho''=-\frac{1}{1+\rho}=-2-(a+1)\rho+\rho^2.$$ Let us denote $\Delta\coloneqq a^2+3a+9$. We know that $\mathrm{disc}(f)=\Delta^2$. If $\Delta$ is square-free, then $\mathcal{O}_K=\mathbb{Z}[\rho]$. The converse does not hold in general. 

We say that $K$ is monogenic if there is an element $\gamma\in\mathcal{O}_K$ such that $\mathcal{O}_K=\mathbb{Z}[\gamma]$. In that case, $\{1,\gamma,\gamma^2\}$ is a power integral basis of $K$. Denote by $\mathfrak{c}$ the conductor of $K$. The monogenity of $K$ is characterized in \cite{kashiosekigawa} in the following way:

\begin{teo}\cite[Corollary 1.6]{kashiosekigawa}\label{charactmonog} The following are equivalent:
\begin{itemize}
    \item[1.] The field $K$ is monogenic.
    \item[2.] We have $a\in\{-1,0,1,2,3,5,12,54,66,1259,2389\}$ or $\frac{\Delta}{\mathfrak{c}}$ is a cube.
    \item[3.] We have $a\in\{-1,0,1,2,3,5,12,54,66,1259,2389\}$ or $a\not\equiv3,21\,(\mathrm{mod}\,27)$ and $v_p(\Delta)\not\equiv2\,(\mathrm{mod}\,3)$ for all $p\neq3$.
\end{itemize} In such a case, a power integral basis of $K$ is generated by $\gamma=\frac{\rho-b}{n}$, where $n=\sqrt[3]{\frac{\Delta}{27}}$ and $b\in\mathbb{Z}$ satisfies $b\equiv\frac{a}{3}\,(\mathrm{mod}\,n)$ if $3\mid n$ and $3a\equiv b\,(\mathrm{mod}\,n)$ otherwise.
\end{teo}

\begin{rmk}\normalfont\label{rmkexcpmonog} The list of possible values for $a$ in the second and the third statement in Theorem~\ref{charactmonog} are the ones for which the corresponding simplest cubic field $K$ is also defined by some other value $a'$ in the same list. Hence, its presence is due to the fact that for such a simplest cubic field $K$, one of the values satisfies the conditions of monogenity and the other one does not, resulting that $K$ is monogenic.
\end{rmk}


Since $K$ is totally real, its conductor $\mathfrak{c}$ is positive.  Moreover, it is related to the discriminant by the so called conductor-discriminant formula, $\mathfrak{c}^2=\mathrm{disc}(K)$. Taking also into account the well known formula \begin{equation}\label{eq:reldiscs}\mathrm{disc}(f)=[\mathcal{O}_K:\mathbb{Z}[\rho]]^2\mathrm{disc}(K),\end{equation} we deduce that $\mathfrak{c}$ divides $\Delta$ and $$[\mathcal{O}_K:\mathbb{Z}[\rho]]=\frac{\Delta}{\mathfrak{c}}.$$

\begin{rmk}\normalfont Note that $K$ may be monogenic even if $\mathcal{O}_K\neq\mathbb{Z}[\rho]$, namely whenever $[\mathcal{O}_K:\mathbb{Z}[\rho]]$ is a non-trivial cube. This means that $\mathcal{O}_K=\mathbb{Z}[\gamma]$ for some $\gamma\neq\rho$. However, we will usually refer to the case $\mathcal{O}_K=\mathbb{Z}[\rho]$ as \textit{the monogenic case}.
\end{rmk}

In order to compute effectively the generalized module index $[\mathcal{O}_K:\mathbb{Z}[\rho]]$, we denote $\Delta=bc^3$ with $b,c>0$ coprime integers and $b$ cube-free.
Then:

\begin{pro}\cite[Remark 1.5]{kashiosekigawa}\label{formconductor} The conductor can be described as follows: $$\mathfrak{c}=\begin{cases}\prod_{p\mid b}p & \hbox{if }3\nmid a\hbox{ or }a\equiv12\,(\mathrm{mod}\,27),\\ 3^2\prod_{p\mid b,p\neq3}p & \hbox{otherwise}.\end{cases}$$
\end{pro}

\begin{rmk}\normalfont\label{rmk:3valconductor} If $3\nmid a$, we have that $\Delta=a^2+3a+9\equiv1\,(\mathrm{mod}\,3)$, and if $a\equiv12\,(\mathrm{mod}\,27)$, then $\Delta\equiv0\,(\mathrm{mod}\,27)$. In both cases, we see that $3\nmid\mathfrak{c}$. We deduce that the conductor of a simplest cubic field $K$ satisfies either that $3\nmid\mathfrak{c}$ or $v_3(\mathfrak{c})=2$.
\end{rmk}

\section{Non-monogenic simplest cubic fields with index $3$}\label{sect:scfprime}

In this part, we will characterize the simplest cubic fields $K=\mathbb{Q}(\rho)$ such that $[\mathcal{O}_K:\mathbb{Z}[\rho]]=3$. These are the simplest cubic fields which have an integral basis
\begin{equation}\label{intbasis}
    B_3\coloneqq\Big\{1,\rho,\frac{1+\rho+\rho^2}{3}\Big\}.
\end{equation}
Indeed, it may be checked, for instance using mathematical software, that $\mathrm{disc}\Big(1,\rho,\frac{1+\rho+\rho^2}{3}\Big)=\Big(\frac{\Delta}{3}\Big)^2$. From the formula \eqref{eq:reldiscs} we find that $B_3$ is an integral basis for $K$.

We can easily check that the minimal polynomial of $\frac{1+\rho+\rho^2}{3}$ is
\[
x^3-\frac{a^2+3a+9}{3}x^2+\frac{2a^2+6a+18}{3^2}x-\frac{a^2+3a+9}{3^3}
\]
Suppose that $a\equiv3$ or $21$ $(\mathrm{mod}\,27)$. A direct computation shows that $\Delta=a^2+3a+9$ is divisible by $27$, implying that the above polynomial has integer coefficients and, thus, $\frac{1+\rho+\rho^2}{3}$ is an algebraic integer.  

In fact, the simplest cubic fields with integral basis $B_3$ can be characterized as follows.

\begin{pro} The simplest cubic fields $K$ with integral basis $B_3$ are the ones for which $a\equiv3,21\,(\mathrm{mod}\,27)$, $a>12$ and $\frac{\Delta}{27}$ is square-free.
\end{pro}
\begin{proof}
Assume that $K$ has an integral basis as in the statement. Then $\mathrm{disc}(K)=\mathrm{disc}(1,\rho,\frac{1+\rho+\rho^2}{3})=\Big(\frac{\Delta}{3}\Big)^2$, so $\Delta=3\mathfrak{c}$ and $K$ is not monogenic. In particular, $3$ divides $\Delta=a^2+3a+9$, so $3\mid a$ and $9\mid\Delta$.  Thus $3\mid\mathfrak{c}$, so Remark \ref{rmk:3valconductor} gives that $9\mid\mathfrak{c}$, whence $27\mid\Delta$. Since $p^2$ does not divide $\mathfrak{c}$ for every prime $p\neq 3$, we obtain that $\frac{\Delta}{27}$ is square-free. On the other hand, by Theorem \ref{charactmonog} (3), the non-monogenity of $K$ gives that $a\equiv3$ or $21$ $(\mathrm{mod}\,27)$. Now, since $K$ is not monogenic, $K$ is defined by a single value of $a$. That is, $a\notin\{-1,0,1,2,3,5,12,54,66,1259,2389\}$. But in this list no element greater than $12$ is congruent to $3$ or $21$ mod $27$, so $a>12$.

Conversely, let $K$ be a simplest cubic field defined by $a$ such that $a\equiv3,21\,(\mathrm{mod}\,27)$ (so that $\Delta$ is divisible by $27$), $a>12$ and $\frac{\Delta}{27}$ is square-free. Then, we can write $\Delta=27d$ with $d$ square-free and from Proposition \ref{formconductor} we see that $\mathfrak{c}=9d$, so $\delta=3$. In particular, $K$ is not monogenic. On the other hand, $\frac{1+\rho+\rho^2}{3}$ is an algebraic integer because the coefficients of its minimal polynomial over $\mathbb{Q}$ are integers. Now, $\mathrm{disc}(1,\rho,\frac{1+\rho+\rho^2}{3})=\Big(\frac{\Delta}{3}\Big)^2=\mathfrak{c}^2$, so $K$ has integral basis $\Big\{1,\rho,\frac{1+\rho+\rho^2}{3}\Big\}$.
\end{proof}

Now, we will analyze some properties of this family of simplest cubic fields. First, the condition that $\frac{\Delta}{27}$ is square-free is actually needed.

\begin{rmk}\normalfont Not all simplest cubic fields with $a\equiv3,21\,(\mathrm{mod}\,27)$ satisfy that $\frac{\Delta}{27}$ is square-free. For instance, if $a=678\equiv3\,(\mathrm{mod}\,27)$, then $v_7(\Delta)=2$. In that case, the integral basis of $K$ is $$\Big\{1,\rho,\frac{4+10\rho+\rho^2}{21}\Big\}.$$
\end{rmk}

Let $K$ be a simplest cubic field with integral basis $B_3$ and write $a=27t+r$, where $r\in\{3,21\}$. Then, $$\frac{\Delta}{27}=\begin{cases}
27t^2+9t+1 &\hbox{ if }a\equiv3\,(\mathrm{mod}\,27), \\
27t^2+45t+19 &\hbox{ if }a\equiv21\,(\mathrm{mod}\,27).
\end{cases}$$ These are quadratic polynomials in $t$, and it can be easily shown that they are irreducible. Moreover, for example, their value for $t=0$ is square-free. So, by the result of Nagel \cite{nagel}, they take infinitely many square-free values. Then, we have infinitely many simplest cubic fields within this family.

This family of simplest cubic fields is just the family considered by C\'anovas Orvay in \cite[Theorem 3]{canovasorvay}. 

\begin{lema} The simplest cubic fields $K$ with integral basis $B_3$ are just the fields of the form $K=\mathbb{Q}(\theta)$ with $\theta^3-p\theta+pq=0$, where $p=9d$ with $d\in\mathbb{Z}$ square-free and $q\in\mathbb{Z}$ is such that $q>2$, $3\nmid q$ and $4p-27q^2=9$.
\end{lema}
\begin{proof}
Let $K=\mathbb{Q}(\rho)$ be a simplest cubic field as in the statement. Setting $\theta=-\rho+\frac{a}{3}$, we have $K=\mathbb{Q}(\theta)$. Let $p=\frac{a^2+3a+9}{3}$ and $q=\frac{2a+3}{9}$. Since $\Delta=27d$, we have \begin{equation*}
    \begin{split}
        \theta^3-p\theta+pq&=-\rho^3+\rho^2a-\rho\frac{a^2}{3}+\frac{a^3}{27}+\Big(\frac{a^2}{3}+a+3\Big)\rho-\frac{(a^2+3a+9)(a-3)}{27}\\&=-\rho^3+a\rho^2+(a+3)\rho+\frac{a^3}{27}-\frac{a^3-27}{27}\\&=-\rho^3+a\rho^2+(a+3)\rho+1=0,
    \end{split}
\end{equation*} proving that $\theta$ has minimal polynomial $x^3-px+pq$ over $\mathbb{Q}$. Moreover, we have that $$4p-27q^2=4\Big(\frac{a^2}{3}+a+3\Big)-\frac{1}{3}\Big(4a^2+12a+9\Big)=9.$$ Conversely, let $p=9d$ with $d\in\mathbb{Z}$ square-free, $q\in\mathbb{Z}$ with $q>2$ and $3\nmid q$ and $4p-27q^2=9$. Let $K=\mathbb{Q}(\theta)$ with $\theta^3-p\theta+pq=0$, and define $a=\frac{9q-3}{2}$. From the equality with $p$ and $q$ we deduce that $q\equiv 1$ or $5$ $(\mathrm{mod}\,6)$, so $a\in\mathbb{Z}$ and $a\equiv3$ or $21$ $(\mathrm{mod}\,27)$. Let us call $\rho=-\theta+\frac{a}{3}$. Going backwards in the chain of equalities above, we find that $\rho$ has irreducible polynomial $x^3-ax^2-(a+3)x-1$ over $\mathbb{Q}$.
\end{proof}

\begin{coro}\label{coro:fundamentalunits} Let $K=\mathbb{Q}(\rho)$ be a simplest cubic field with integral basis $B_3$. Then $\{\rho,\rho'\}$ is a system of fundamental units of $K$.
\end{coro}
\begin{proof}
Let $\theta_1=4\frac{p}{3}-3q\theta-2\theta^2$, $\sigma=\frac{\theta_1}{3}$, $\tau=\frac{1}{q}\Big(\frac{q-1}{2}+\frac{q-1}{2}\sigma+\sigma^2\Big)$ and $\mu=\sigma+\tau$. In \cite[Theorem 3 (iii)]{canovasorvay}, it is stated that $\{\mu,\mu'\}$ is a system of fundamental units of $K$. Now, we may check directly that $\mu=\rho'+1$.
Since Galois automorphisms preserve systems of fundamental units, we deduce that $\{\rho+1,\rho'+1\}$ is a system of fundamental units. Finally, since $\rho+1=-\rho\rho'$ and $\rho'+1=-\frac{1}{\rho}$,  $\{\rho,\rho'\}$ is a system of fundamental units of $K$.
\end{proof}

\section{The parallelepipeds of lattice points}

In this section, we turn to the problem of finding the indecomposable integers in a simplest cubic field $K$ for cases when  $[\mathcal{O}_K:\mathbb{Z}[\rho]]=3$. We aim to find an analogue of Proposition \ref{pro:ppidsmonogscf} for this case. We will follow the method at \cite{kalatinkova} combined with a slight variation.

Let us fix the power basis $\{1,\rho,\rho^2\}$, which is not integral in our case. We have that $\mathcal{O}_K$ embeds in $\frac{1}{3}\mathbb{Z}^3$ through $$\begin{array}{rccl}
    \tau\colon & K & \longrightarrow & \mathbb{R}^3, \\
     & \alpha_1+\alpha_2\rho+\alpha_3\rho^2 & \longmapsto & (\alpha_1,\alpha_2,\alpha_3).
\end{array}$$ Now, let us call $\varepsilon_1=\rho^2$ and $\varepsilon_2=(\rho'')^{-2}$. From \cite[Corollary 2]{thomasvasquez} we know that the pair $(\varepsilon_1,\varepsilon_2)$ is proper (note that conditions (1)-(3) in Definition \ref{def:properpair} do not depend on whether we have $\mathcal{O}_K=\mathbb{Z}[\rho]$). Since $3\mathcal{O}_K\subseteq\mathbb{Z}[\rho]$, the indecomposables of $K$ lie, up to multiplication by totally positive units, at the intersection of $\frac{1}{3}\mathbb{Z}^3$ with the parallelepipeds \begin{equation}\label{eq:parallelepipeds}
    \mathcal{D}(1,\rho^2,1+2\rho+\rho^2),\quad\mathcal{D}(1,\rho^2,-1-a-(a^2+3a+3)\rho+(a+2)\rho^2).
\end{equation}
From now on, these will be referred to as the first parallelepiped and the second parallelepiped, respectively.

\begin{rmk}\normalfont There might be elements in $\frac{1}{3}\mathbb{Z}[\rho]$ that do not lie in $\mathcal{O}_K$. In other words, there might be some lattice points that are not algebraic integers.
\end{rmk}

Then, the problem turns to the determination of the points at the intersection of any of the parallelepipeds with $\frac{1}{3}\mathbb{Z}^3$. For simplicity, these will be also referred to as lattice points.

\subsection{Lattice points in the first parallelepiped}

We shall find the lattice points in the first of the above parallelepipeds.
We look for $t_1,t_2,t_3\in[0,1]$ such that for some $(m,n,o)\in\mathbb{Z}^3$, $$t_1+t_2\rho^2+t_3(1+2\rho+\rho^2)=\frac{1}{3}(m+n\rho+o\rho^2).$$ This gives rise to the system of equations 
\[
t_1+t_3=\frac{1}{3}m, \qquad
2t_3=\frac{1}{3}n, \qquad
t_2+t_3=\frac{1}{3}o.
\]
The second equation gives that $t_3=\frac{n}{6}$ with $0\leq n\leq 6$. Carrying this to the first equation, $t_1=\frac{2m-n}{6}$, with $\frac{n}{2}\leq m\leq\frac{n}{2}+3$. As for the third one, $t_2=\frac{2o-n}{6}$, with $\frac{n}{2}\leq o\leq\frac{n}{2}+3$. We plug these values in the parametric equation, obtaining the points $$\frac{1}{3}(m+n\rho+o\rho^2),\quad0\leq n\leq6,\quad\frac{n}{2}\leq m,o\leq\frac{n}{2}+3,$$ which are candidates to be indecomposables of $K$. From these element, those ones which are non-unit algebraic integers are of the form $t\frac{1+\rho+\rho^2}{3}$ where $1\leq t\leq 5$. We immediately see that if $t\geq 2$, then $t\frac{1+\rho+\rho^2}{3}$ is decomposable (which is also true for $1+\rho+\rho^2$, which is indecomposable in $\Z[\rho]$). In the case of $g_3=\frac{1+\rho+\rho^2}{3}$, we will show that this element is indeed indecomposable in $\O_K$ (see Proposition \ref{proindecs12}).  

\subsection{Lattice points in the second parallelepiped}

For the second parallelepiped, we use the integral basis of a simplest cubic field in this family, which is $$B_3=\Big\{1,\rho,\frac{1+\rho+\rho^2}{3}\Big\}.$$ We just rewrite the nodes of the second parallelepiped with respect to this integral basis. Call $g_1=1$, $g_2=\rho$, $g_3=\frac{1+\rho+\rho^2}{3}$. Then, the second parallelepiped becomes $$\mathcal{D}(g_1,-g_1-g_2+3g_3,-(2a+3)g_1-(a^2+4a+5)g_2+3(a+2)g_3),$$ and we want to determine its lattice points. That is, we look for $t_1,t_2,t_3\in[0,1]$ and integer numbers $m,n,o\in\mathbb{Z}$ such that $$t_1g_1+t_2(-g_1-g_2+3g_3)+t_3(-(2a+3)g_1-(a^2+4a+5)g_2+3(a+2)g_3)=mg_1+ng_2+og_3.$$ This is equivalent to the following system of equations
    \begin{align}
        t_1-t_2-(2a+3)t_3&=m, \notag \\
        -t_2-(a^2+4a+5)t_3&=n, \label{systemt_i} \\
        3t_2+3(a+2)t_3&=o. \notag
    \end{align}

From the second and the third equation we obtain $-3(a^2+3a+3)t_3=o+3n$. As in the monogenic case, the choices $t_3=0$ or $t_3=3(a^2+3a+3)$ give either totally positive units or their sums. Therefore, \begin{equation}\label{eq:t3val}
    t_3=\frac{u}{3(a^2+3a+3)},\quad1\leq u\leq3(a^2+3a+3)-1.
\end{equation}

From the first and the second equation we obtain that $t_1+(a^2+2a+2)t_3=m-n$, that is, $$t_1=m-n-u\Big(\frac{1}{3}-\frac{a+1}{3(a^2+3a+3)}\Big).$$ Write $u=3w+s$, where $0\leq w\leq a^2+3a+2$, $0\leq s\leq2$ and $s>0$ if $w=0$. Then $$t_1=m-n-w+\frac{w(a+1)}{a^2+3a+3}-s\Big(\frac{1}{3}-\frac{a+1}{3(a^2+3a+3)}\Big).$$ Write $w=v(a+2)+r$, where $0\leq v,r\leq a+1$ and $r=0$ if $v=a+1$. Then $$t_1=m-n-w+v+\frac{r(a+1)-v}{a^2+3a+3}-s\frac{a^2+2a+2}{3(a^2+3a+3)}.$$ Let us define $$e_1=\Big\lceil-\Big(\frac{r(a+1)-v}{a^2+3a+3}-s\frac{a^2+2a+2}{3(a^2+3a+3)}\Big)\Big\rceil.$$ Hence, $t_1\in[0,1]$ if and only if $m-n=w-v+e_1$. We can easily write $$e_1=\Big\lceil\frac{s(a^2+2a+2)-3(r(a+1)-v)}{3(a^2+3a+3)}\Big\rceil.$$ Then, we have that
\begin{itemize}
    \item For $s=0$, $e_1=1$ if $r=0$, and $e_1=0$ otherwise.
    \item For $s=1$, $e_1=1$ if $3(r(a+1)-v)\leq a^2+2a$, and $e_1=0$ otherwise.
    \item For $s=2$, $e_1=1$ if $3(r(a+1)-v)\leq2a^2+4a+3=2a(a+2)+3$, and $e_1=0$ otherwise.
\end{itemize}
\begin{rmk}\normalfont In the case $s=0$, $e_1=0$ would be also possible if $v=0$ and $r=0$, but this corresponds to the zero point, which is not totally positive.
\end{rmk}

On the other hand, replacing the expression \eqref{eq:t3val} of $t_3$ in the second equation of \eqref{systemt_i} gives $$t_2=-n-u\Big(\frac{1}{3}+\frac{a+2}{3(a^2+3a+3)}\Big).$$ Now, replacing $u=3w+s$, $$t_2=-n-w-\frac{w(a+2)}{a^2+3a+3}-s\Big(\frac{1}{3}+\frac{a+2}{3(a^2+3a+3)}\Big).$$ Write $w=t(a+1)+l$, with $0\leq t\leq a+2$, $0\leq l\leq a$ and $l=0$ if $t=a+2$. Now, $$t_2=-n-w-t+\frac{t-l(a+2)}{a^2+3a+3}-s\frac{a^2+4a+5}{3(a^2+3a+3)}.$$  Let us call $$e_2=\Big\lceil-\Big(\frac{t-l(a+2)}{a^2+3a+3}-s\frac{a^2+4a+5}{3(a^2+3a+3)}\Big)\Big\rceil.$$ Then $t_2\in[0,1]$ if and only if $n=-w-t-e_2$. An easy calculation shows that $$e_2=\Big\lceil\frac{s(a^2+4a+5)+3(l(a+2)-t)}{3(a^2+3a+3)}\Big\rceil.$$ Thus, in this case:
\begin{itemize}
    \item For $s=0$, $e_2=0$ if $l=0$, and $e_2=1$ otherwise.
    \item For $s=1$, $e_2=1$ if $3(l(a+2)-t)\leq2a^2+5a+3=(2a+3)(a+1)$, and $e_2=2$ otherwise.
    \item For $s=2$, $e_2=1$ if $3(l(a+2)-t)\leq a^2+a-3=a(a+1)-3$, and $e_2=2$ otherwise.
\end{itemize} 
\begin{rmk}\normalfont In the case $s=0$, $e_2=0$ would be also possible if $l=1$ and $t=a+2$, but we have excluded that combination by taking $l=0$ if $t=a+2$.
\end{rmk}
Moreover, since $-u=o+3n$, we have that $o=-u+3(w+t+e_2)$.

This gives the lattice points $$-(v+t-e_1+e_2)g_1-(w+t+e_2)g_2+(-u+3(w+t+e_2))g_3.$$ Since $u=3w+s$, we may rewrite them as $$-(v+t-e_1+e_2)g_1-(w+t+e_2)g_2+(-s+3(t+e_2))g_3.$$

With respect to the power (non-integral) basis, these become \begin{equation*}
    -v-\frac{s}{3}+e_1-\Big(w+\frac{s}{3}\Big)\rho+\Big(t-\frac{s}{3}+e_2\Big)\rho^2.
\end{equation*}

\begin{pro}\label{procandidates3} Let $K$ be a simplest cubic field with $a\equiv3$ or $21\,(\mathrm{mod}\,27)$, $a>12$ and $\frac{\Delta}{27}$ square-free. The lattice points of $K$ (and therefore non-unit candidates to be indecomposables of $K$; up to multiplication by totally positive units) are the point $\frac{1+\rho+\rho^2}{3}$ from the first parallelepiped, and the points $$-(v+t-e_1+e_2)g_1-(w+t+e_2)g_2+(-s+3(t+e_2))g_3$$ $$=-(v-e_1)-w\rho+(t+e_2)\rho^2-sg_3$$ from the second one, where:  
\begin{itemize}
    \item $0\leq w\leq a^2+3a+2$,
    \item $0\leq s\leq2$, and $s>0$ if $w=0$,
    \item $v$ and $r$ are defined from $w=v(a+2)+r$, $0\leq v,r\leq a+1$ and $r=0$ if $v=a+1$,
    \item $t$ and $l$ are defined from $w=t(a+1)+l$, $0\leq t\leq a+2$, $0\leq l\leq a$, and $l=0$ if $t=a+2$,
    \item $e_1=\Big\lceil\frac{s(a^2+2a+2)+3(v-r(a+1))}{3(a^2+3a+3)}\Big\rceil$,
    \item $e_2=\Big\lceil\frac{s(a^2+4a+5)+3(l(a+2)-t)}{3(a^2+3a+3)}\Big\rceil$.
\end{itemize}
\end{pro}

Note that each lattice point $\alpha$ in the second parallelepiped is completely determined from $s$ and the integers $v$ and $r$ as in Proposition \ref{procandidates3} such that $\alpha=-(v-e_1)-(v(a+2)+r)\rho+(t+e_2)\rho^2-sg_3$, because such a choice determines completely the pair $(t,l)$. From now on, we will denote $\alpha=\alpha_s(v,r)$ or, for $w=v(a+2)+r$, $\alpha=\alpha_s(w)$.

\section{Indecomposables when $[\O_K:\Z[\rho]]=3$}\label{sec:indec}

In this part we will prove Theorem \ref{teolistindec}, i.e, we will find all the indecomposables of a simplest cubic field with integral basis $B_3$ up to multiplication by totally positive units. In the previous sections, we found all the lattice points in the two parallelepipeds, and now we determine which ones are indecomposable. 

To complete this task, we will proceed as follows:
\begin{enumerate}
    \item First, we specify algebraic integers belonging to the second parallelepiped (Subsections \ref{sects=0}, \ref{sects=1} and \ref{sects=2}). \label{pt:1}
    \item Then we prove that elements in Theorem \ref{teolistindec} are indecomposable (Subsection \ref{subsubsec:indecomposable}). 
    \item For the elements in (\ref{pt:1}) not included in Theorem \ref{teolistindec}, we 
    \begin{enumerate}
        \item first exclude some of those elements which are just unit multiples of conjugates of other ones (Subsection \ref{sect:t1t2}) \label{pt:11} 
        \item and then find decompositions as a sum of two totally positive algebraic integers for elements not excluded in (\ref{pt:11}) (Subsection \ref{sec:decspoints}). 
    \end{enumerate}
\end{enumerate}

First of all, let us identify the lattice points for the case $w=0$. In this case, $s=0$ would give the zero point, which is excluded. Otherwise, since we have $v=t=r=l=0$, then $e_1=e_2=1$. We obtain the points \begin{equation}\label{lowpoints}
    \alpha_2(0,0)=-g_2+g_3,\quad\alpha_1(0,0)=-g_2+2g_3.
\end{equation}

The case $w=(a+1)(a+2)$ is easy as well. This value of $w$ corresponds to $t=a+2$, $v=a+1$ and $r=l=0$. For $s=0$, we see that $e_1=1$ and $e_2=0$. We obtain the point $$\alpha_0(a+1,0)=-(2a+2)-(a+2)^2g_2+3(a+2)g_3=-a-(a+1)(a+2)\rho+(a+2)\rho^2.$$ For $s\neq0$, we have $e_1=e_2=1$, giving the point \begin{equation}\label{highpoints}
    \begin{split}
        \alpha_s(a+1,0)&=-(2a+3)-((a+2)^2+1)g_2+(-s+3(a+3))g_3\\&=-a-(a+1)(a+2)\rho+(a+3)\rho^2-sg_3.
    \end{split}
\end{equation}

Fixed a value of $s$, we will distribute the lattice points $\alpha_s(v,r)$ in four situations depending on the value of $w=v(a+2)+r$.

\begin{itemize}
    \item[(a)] $x(a+2)<w<(x+1)(a+1)$, $0\leq x\leq a-1$, which corresponds to $1\leq r\leq a-v$. In this case, $v=t=x$ and $l=r+v$. 
    \item[(b)] $w=(x+1)(a+1)$, $0\leq x\leq a$. In this case, $v=x$, $t=v+1$, $r=a-v+1$ and $l=0$.
    \item[(c)] $(x+1)(a+1)<w<(x+1)(a+2)$, $1\leq x\leq a$, which corresponds to $v=x$, $a-v+2\leq r\leq a+1$, $t=v+1$ and $l=r-(a-v+1)$.
    \item[(d)] $w=(x+1)(a+2)$, $0\leq x\leq a-1$. Then $v=x+1$, $r=0$ and $t=l=v$.
\end{itemize}

It may be checked easily that these cases, together with $\alpha_s(0,0)$, $s\in\{1,2\}$ and $\alpha_s(a+1,0)$, $s\in\{0,1,2\}$, cover exactly once all points in Proposition \ref{procandidates3}.

\subsection{The case $s=0$}\label{sects=0}

If in Proposition \ref{procandidates3} we take $s=0$, this corresponds to the intersection of the second parallelepiped with $\mathbb{Z}[\rho]$. Then, we are left just with the lattice points in $\mathbb{Z}[\rho]$, which have been already determined in \cite{kalatinkova}. In other words, the points $\alpha_s(v,r)$ with $s=0$ are exactly the ones that have integer coordinates with respect to the power basis $\{1,\rho,\rho^2\}$. With our notation, these are $$\begin{array}{ll}
    -v-w\rho+(v+1)\rho^2, & \hbox{if }v(a+2)<w\leq (v+1)(a+1)\hbox{ and }0\leq v\leq a, \\
    -v-w\rho+(v+2)\rho^2, & \hbox{if }(v+1)(a+1)<w<(v+1)(a+2)\hbox{ and }0\leq v\leq a, \\
    -(v-1)-w\rho+(v+1)\rho^2, & \hbox{if }w=v(a+2)\hbox{ and }0\leq v\leq a+1. \\
\end{array}$$ While the first line corresponds to situations (a) and (b), the second (resp. third) one corresponds to situation (c) (resp. (d)).

\subsection{The case $s=1$} \label{sects=1}

Assume that $s=1$. Recall that $w=v(a+2)+r=t(a+1)+l$. Now, we have that $$e_1=\begin{cases}
1 & \hbox{if }3(r(a+1)-v)\leq a(a+2), \\
0 & \hbox{otherwise},
\end{cases}$$ $$e_2=\begin{cases}
1 & \hbox{if }3(l(a+2)-t)\leq (2a+3)(a+1), \\
0 & \hbox{otherwise}.
\end{cases}$$

\subsubsection*{Situation (a)}

We first take $x(a+2)<w<(x+1)(a+1)$, $0\leq x\leq a-1$. We know that $x=v$. Then: \begin{equation*}
    \begin{split}
        &3(r(a+1)-v)\leq a(a+2) \\
        \Longleftrightarrow\;&r(a+1)\leq\frac{a}{3}(a+2)+v \\
        \Longleftrightarrow\;&r\leq\Big\lfloor\frac{a(a+2)}{3(a+1)}+\frac{v}{a+1}\Big\rfloor=\Big\lfloor\frac{a}{3}\Big(1+\frac{1}{a+1}\Big)+\frac{v}{a+1}\Big\rfloor=\frac{a}{3}+\Big\lfloor\frac{3v+a}{3(a+1)}\Big\rfloor.
    \end{split}
\end{equation*} Now, $3v+a<3(a+1)$ if and only if $v<\frac{2a}{3}+1$. We deduce that $$e_1=\begin{cases}
    1 & \hbox{if }0\leq v\leq\frac{2a}{3}\hbox{ and }1\leq r\leq\frac{a}{3} \\
     & \hbox{or }\frac{2a}{3}+1\leq v\leq a-1, \\
    0 &\hbox{if }0\leq v\leq\frac{2a}{3}-1\hbox{ and }\frac{a}{3}+1\leq r\leq a-v.
\end{cases}$$ On the other hand, \begin{equation*}
    \begin{split}
        &3(l(a+2)-t)=3(v(a+1)+r(a+2))\leq (2a+3)(a+1) \\
        \Longleftrightarrow\;&r(a+2)\leq\Big(\frac{2a}{3}+1-v\Big)(a+1) \\
        \Longleftrightarrow\;&r\leq\Big\lfloor\Big(\frac{2a}{3}+1-v\Big)\frac{a+1}{a+2}\Big\rfloor=\frac{2a}{3}+1-v+\Big\lfloor\frac{1}{a+2}\Big(v-\Big(\frac{2a}{3}+1\Big)\Big)\Big\rfloor.
    \end{split}
\end{equation*} Note that the right-side member needs to be positive, i.e. $v<\frac{2a}{3}+1$. Under this assumption, we have that $e_2=1$ if and only if $1\leq r\leq\frac{2a}{3}-v$ (since $\frac{2a}{3}+1-v<a+2$). We deduce that $$\begin{cases}
    e_2=1 & \hbox{if }0\leq v\leq\frac{2a}{3}\hbox{ and }1\leq r\leq\frac{2a}{3}-v, \\
    e_2=2 & \hbox{if }0\leq v\leq\frac{2a}{3}\hbox{ and }\frac{2a}{3}+1-v\leq r\leq a-v \\
     & \hbox{or }\frac{2a}{3}+1\leq v\leq a. \\
\end{cases}$$ Joining all the information together, we have that:
\begin{itemize}
    \item $e_1=e_2=1$ exactly in the cases:
    \begin{itemize}
        \item $0\leq v\leq\frac{a}{3}\hbox{ and }1\leq r\leq\frac{a}{3}$,
        \item $\frac{a}{3}+1\leq v\leq\frac{2a}{3}-1\hbox{ and }1\leq r\leq\frac{2a}{3}-v$.
    \end{itemize} These correspond to the points $$-2vg_1-(w+v+1)g_2+(3v+2)g_3=-(v-1)-w\rho+(v+1)\rho^2-g_3.$$
    \item $e_1=0$ and $e_2=1$ if and only if $0\leq v\leq\frac{a}{3}-1$ and $\frac{a}{3}+1\leq r\leq\frac{2a}{3}-v$. These correspond to the points $$-(2v+1)g_1-(w+v+1)+(3v+2)g_3=-v-w\rho+(v+1)\rho^2-g_3.$$
    \item $e_1=1$ and $e_2=2$ exactly in the cases:
    \begin{itemize}
        \item $\frac{a}{3}+1\leq v\leq\frac{2a}{3}$ and $\frac{2a}{3}-v+1\leq r\leq\frac{a}{3}$,
        \item $\frac{2a}{3}+1\leq v\leq a-1$.
    \end{itemize} These correspond to the points $$-(2v+1)g_1-(w+v+2)g_2+(3v+5)g_3=-(v-1)-w\rho+(v+2)\rho^2-g_3.$$
    \item $e_1=0$ and $e_2=2$ exactly in the cases:
    \begin{itemize}
        \item $0\leq v\leq\frac{a}{3}$ and $\frac{2a}{3}-v+1\leq r\leq a-v$,
        \item $\frac{a}{3}+1\leq v\leq\frac{2a}{3}-1$ and $\frac{a}{3}+1\leq r\leq a-v$.
    \end{itemize} These correspond to the points $$-(2v+2)g_1-(w+v+2)g_2+(3v+5)g_3=-v-w\rho+(v+2)\rho^2-g_3.$$
\end{itemize}

\subsubsection*{Situation (b)}

Suppose that $w=(x+1)(a+1)$ with $0\leq x\leq a$ and $x=v$. Now, we have \begin{equation*}
    \begin{split}
        &3(r(a+1)-v)=3((a-v)(a+2)+1)\leq a(a+2) \\
        \Longleftrightarrow\; & (a-v)(a+2)+1\leq\frac{a}{3}(a+2) \\
        \Longleftrightarrow\; & a-v+\frac{1}{a+2}\leq\frac{a}{3} 
        \\
        \Longleftrightarrow\; & a-v+1\leq\frac{a}{3} \\
        \Longleftrightarrow\; & v\geq\frac{2a}{3}+1.
    \end{split}
\end{equation*} Therefore, $e_1=1$ if $v\geq\frac{2a}{3}+1$, and $e_1=0$ otherwise. On the other hand, since $l=0$ we always have that $e_2=1$. Then:
\begin{itemize}
    \item $e_1=1$ and $e_2=1$ if and only if $\frac{2a}{3}+1\leq v\leq a$. This corresponds to the points $$-(2v+1)g_1-(w+v+2)g_2+(3v+5)g_3=-(v-1)-w\rho+(v+2)\rho^2-g_3.$$
    \item $e_1=0$ and $e_2=1$ if and only if $0\leq v\leq\frac{2a}{3}$. This corresponds to the points $$-(2v+2)g_1-(w+v+2)g_2+(3v+5)g_3=-v-w\rho+(v+2)\rho^2-g_3.$$
\end{itemize}

\subsubsection*{Situation (c)}

Next, assume that $(x+1)(a+1)+1\leq w<(x+1)(a+2)$ with $1\leq x\leq a$ and $x=v$. Arguing as in the situation (a), we see that $3(r(a+1)-v)\leq a(a+2)$ if and only if $r\leq\frac{a}{3}+\Big\lfloor\frac{3v+a}{3(a+1)}\Big\rfloor$. Since $\frac{a}{3}+1\geq a-v+2$ if and only if $v\geq\frac{2a}{3}+1$, we deduce that $$e_1=\begin{cases}
    1 & \hbox{if }\frac{2a}{3}+1\leq v\leq a\hbox{ and }a-v+2\leq r\leq\frac{a}{3}+1, \\
    0 & \hbox{if }\frac{2a}{3}+1\leq v\leq a\hbox{ and }\frac{a}{3}+2\leq r\leq a+1 \\
     & \hbox{or }1\leq v\leq\frac{2a}{3}. \\
\end{cases}$$ On the other hand, \begin{equation*}
    \begin{split}
        & 3(l(a+2)-t)=3(l(a+2)-(v+1))\leq (2a+3)(a+1) \\
        \Longleftrightarrow\; & l(a+2)\leq\Big(\frac{2a}{3}+1\Big)(a+1)+v+1 \\
        \Longleftrightarrow\; & l\leq\Big(\frac{2a}{3}+1\Big)\Big(1-\frac{1}{a+2}\Big)+\frac{v+1}{a+2}=\frac{2a}{3}+1+\frac{v-\frac{2a}{3}}{a+2} \\
        \Longleftrightarrow\; & r\leq\frac{5a}{3}-v+2+\frac{v-\frac{2a}{3}}{a+2},
    \end{split} 
\end{equation*} as $l=r-(a-v+1)$. Now, we see that $$\Big\lfloor\frac{v-\frac{2a}{3}}{a+2}\Big\rfloor=\begin{cases}
    -1 & \hbox{if }v<\frac{2a}{3}, \\
    0 & \hbox{if }v\geq\frac{2a}{3}.
\end{cases}$$ In the first case, we have that $e_2=1$ when $a-v+2\leq r\leq\mathrm{min}(a+1,\frac{5a}{3}-v+1)=a+1$, so there are no further restrictions on $r$. As for the second one, we have that $e_2=1$ when $a-v+2\leq r\leq\mathrm{min}(a+1,\frac{5a}{3}-v+2)$. This equals $\frac{5a}{3}-v+2$ unless $v=\frac{2a}{3}$, in which case the minimum is $a+1$. We conclude that $$e_2=\begin{cases}
    1 & \hbox{if }1\leq v\leq\frac{2a}{3}, \\
      & \hbox{or }\frac{2a}{3}+1\leq v\leq a\hbox{ and }a-v+2\leq r\leq\frac{5a}{3}-v+2, \\
    2 &\hbox{if }\frac{2a}{3}+1\leq v\leq a\hbox{ and }\frac{5a}{3}-v+3\leq r\leq a+1.
\end{cases}$$ We note that $e_1=1$ implies that $e_2=1$, and consequently the case that $e_1=1$ and $e_2=2$ is not possible.

\begin{itemize}
    \item $e_1=e_2=1$ if and only if $\frac{2a}{3}+1\leq v\leq a\hbox{ and }a-v+2\leq r\leq\frac{a}{3}+1$. This corresponds to the points $$-(2v+1)g_1-(w+v+2)g_2+(3v+5)g_3=-(v-1)-w\rho+(v+2)\rho^2-g_3.$$
    \item $e_1=0$ and $e_2=1$ exactly in the cases:
    \begin{itemize}
        \item $1\leq v\leq\frac{2a}{3}$.
        \item $\frac{2a}{3}+1\leq v\leq a$ and $\frac{a}{3}+2\leq r\leq\frac{5a}{3}-v+2$.
    \end{itemize}
    This corresponds to the points $$-(2v+2)g_1-(w+v+2)g_2+(3v+5)g_3=-v-w\rho+(v+2)\rho^2-g_3.$$
    \item $e_1=0$ and $e_2=2$ if and only if $\frac{2a}{3}+2\leq v\leq a$ and $\frac{5a}{3}-v+3\leq r\leq a+1$. 
    This corresponds to the points $$-(2v+3)g_1-(w+v+3)g_2+(3v+8)g_3=-v-w\rho+(v+3)\rho^2-g_3.$$
\end{itemize}

\subsubsection*{Situation (d)}

Let us take $w=(x+1)(a+2)$ with $0\leq x\leq a-1$. Now, we have $v=x+1$. Since $r=0$, we always have $3(r(a+1)-v)<a(a+2)$, so $e_1=1$. Moreover, $3(l(a+2)-t)=3v(a+1)\leq(2a+3)(a+1)$ if and only if $v\leq\frac{2a}{3}+1$. Therefore, $e_2=1$ if $v\leq\frac{2a}{3}+1$ and $e_2=2$ otherwise. Then:
\begin{itemize}
    \item $e_1=1$ and $e_2=1$ if and only if $1\leq v\leq\frac{2a}{3}+1$. Together with the point $-g_2+2g_3$ from \eqref{lowpoints} (so that $0\leq v\leq\frac{2a}{3}+1$), this corresponds to the points $$-2vg_1-(w+v+1)g_2+(3v+2)g_3=-(v-1)-w\rho+(v+1)\rho^2-g_3.$$
    \item $e_1=1$ and $e_2=2$ if and only if $\frac{2a}{3}+2\leq v\leq a$. Together with the point from \eqref{highpoints} with $s=1$ (so that $\frac{2a}{3}+2\leq v\leq a+1$), this corresponds to the points $$-(2v+1)g_1-(w+v+2)g_2+(3v+5)g_3=-(v-1)-w\rho+(v+2)\rho^2-g_3.$$
\end{itemize}

\subsection{The case $s=2$} \label{sects=2}

Assume that $s=2$. In this case, $$e_1=\begin{cases}
1 & \hbox{if }3(r(a+1)-v)\leq 2a(a+2)+3, \\
0 & \hbox{otherwise},
\end{cases}$$ $$e_2=\begin{cases}
1 & \hbox{if }3(l(a+2)-t)\leq a(a+1)-3, \\
2 & \hbox{otherwise}.
\end{cases}$$

\subsubsection*{Situation (a)}

We first take $x(a+2)<w<(x+1)(a+1)$, where $x=v$. Then: \begin{equation*}
    \begin{split}
        &3(r(a+1)-v)\leq2a(a+2)+3\\
        \Longleftrightarrow\;&r(a+1)\leq\frac{2a}{3}(a+2)+v+1\\
        \Longleftrightarrow\;&r\leq\frac{2a}{3}\Big(1+\frac{1}{a+1}\Big)+\frac{v+1}{a+1}=\frac{2a}{3}+\frac{v+1+\frac{2a}{3}}{a+1}.
    \end{split}
\end{equation*} We deduce that $$e_1=\begin{cases}
    1 &\hbox{if }0\leq v\leq\frac{a}{3}-1\hbox{ and }1\leq r\leq\frac{2a}{3}\\
    &\hbox{or }\frac{a}{3}\leq v\leq a-1,\\
    0 &\hbox{if }0\leq v\leq\frac{a}{3}-1\hbox{ and }\frac{2a}{3}+1\leq r\leq a-v.
\end{cases}$$ On the other hand, \begin{equation*}
    \begin{split}
        &3(l(a+2)-t)=3(v(a+1)+r(a+2))\leq a(a+1)-3\\
        \Longleftrightarrow\;&r(a+2)\leq\frac{a}{3}(a+1)-1-v(a+1)=\Big(\frac{a}{3}-v\Big)(a+1)-1\\
        \Longleftrightarrow;&r\leq\Big(\frac{a}{3}-v\Big)\Big(1-\frac{1}{a+2}\Big)-\frac{1}{a+2}=\frac{a}{3}-v+\frac{v-\frac{a}{3}-1}{a+2}.
    \end{split}
\end{equation*} Then, $$e_2=\begin{cases}
    1 &\hbox{if }0\leq v\leq\frac{a}{3}-2\hbox{ and }1\leq r\leq\frac{a}{3}-v-1,\\
    2 &\hbox{if }0\leq v\leq\frac{a}{3}-2\hbox{ and }\frac{a}{3}-v\leq r\leq a-v\\
    &\hbox{or }\frac{a}{3}-1\leq v\leq a-1.
\end{cases}$$ 

\begin{itemize}
    \item $e_1=e_2=1$ if and only if $0\leq v\leq\frac{a}{3}-2$ and $1\leq r\leq\frac{a}{3}-v-1$. This corresponds to the points $$-2vg_1-(w+v+1)g_2+(3v+1)g_3=-(v-1)-w\rho+(v+1)\rho^2-2g_3.$$
    \item Since $e_2=1$ implies that $e_1=1$, the case that $e_1=0$ and $e_2=1$ is not possible.
    \item $e_1=1$ and $e_2=2$ exactly in the cases:
    \begin{itemize}
        \item $0\leq v\leq\frac{a}{3}-1$ and $\frac{a}{3}-v\leq r\leq\frac{2a}{3}$.
        \item $\frac{a}{3}\leq v\leq a-1$.
    \end{itemize} These correspond to the points $$-(2v+1)g_1-(w+v+2)g_2+(3v+4)g_3=-(v-1)-w\rho+(v+2)\rho^2-2g_3.$$
    \item $e_1=0$ and $e_2=2$ if and only if $0\leq v\leq\frac{a}{3}-1$ and $\frac{2a}{3}+1\leq r\leq a-v$. In this case, we have the points $$-(2v+2)g_1-(w+v+2)g_2+(3v+4)g_3=-v-w\rho+(v+2)\rho^2-2g_3.$$
\end{itemize}

\subsubsection*{Situation (b)}

Suppose that $w=(x+1)(a+1)$ with $0\leq x\leq a$ and $x=v$. Now, we have \begin{equation*}
    \begin{split}
        &3(r(a+1)-v)=3((a-v)(a+2)+1)\leq2a(a+2)+3\\
        \Longleftrightarrow\;&a(a+2)-v(a+2)+1\leq\frac{2a}{3}(a+2)+1\\
        \Longleftrightarrow\;&v(a+2)\geq\frac{a}{3}(a+2)\\
        \Longleftrightarrow\;&v\geq\frac{a}{3}.
    \end{split}
\end{equation*} We deduce that $e_1=1$ if $v\geq\frac{a}{3}$ and $e_1=0$ otherwise. On the other hand, we always have $e_2=1$ because $l=0$.

\begin{itemize}
    \item $e_1=e_2=1$ if and only if $\frac{a}{3}\leq v\leq a$, giving the points $$-(2v+1)g_1-(w+v+2)g_2+(3v+4)g_3=-(v-1)-w\rho+(v+2)\rho^2-2g_3.$$
    \item $e_1=0$ and $e_2=1$ if and only if $0\leq v\leq\frac{a}{3}-1$, which corresponds to $$-(2v+2)g_1-(w+v+2)g_2+(3v+4)g_3=-v-w\rho+(v+2)\rho^2-2g_3.$$
\end{itemize}

\subsubsection*{Situation (c)}

Next, assume that $(x+1)(a+1)+1\leq w<(x+1)(a+2)$ with $1\leq x\leq a$ and $x=v$. Arguing as in the situation (a), $3(r(a+1)-v)\leq 2a(a+2)+3$ if and only if $r\leq\frac{2a}{3}+\frac{v+1+\frac{2a}{3}}{a+1}$. Assume that $0\leq v\leq\frac{a}{3}-1$. If it was $e_1=1$, we would have that $r\leq\frac{2a}{3}$ but $a-v+2>\frac{2a}{3}$, which is a contradiction. For $v=\frac{a}{3}$, if it was $e_1=1$ we would have that $r\leq\frac{2a}{3}+1$, but $a-v+2=\frac{2a}{3}+2$, giving again a contradiction. Thus, when $v\leq\frac{a}{3}$ we always have $e_1=0$. Now, suppose that $v\geq\frac{a}{3}+1$. Again, $e_1=1$ if and only if $r\leq\frac{2a}{3}+1$, and in this case $a-v+2\geq\frac{2a}{3}+1$. Thus, $$e_1=\begin{cases}
    1 &\hbox{if }\frac{a}{3}+1\leq v\leq a\hbox{ and }a-v+2\leq r\leq\frac{2a}{3}+1,\\
    0 &\hbox{if }\frac{a}{3}+1\leq v\leq a\hbox{ and }\frac{2a}{3}+2\leq r\leq a+1\\
    &\hbox{or }1\leq v\leq\frac{a}{3}.
\end{cases}$$ On the other hand,
\begin{equation*}
    \begin{split}
        & 3(l(a+2)-t)=3(l(a+2)-(v+1))\leq a(a+1)-3 \\
        \Longleftrightarrow\; & l(a+2)\leq\frac{a}{3}(a+1)+v \\
        \Longleftrightarrow\; & l\leq\frac{a}{3}\frac{a+1}{a+2}+\frac{v}{a+2} \\
        \Longleftrightarrow\; & l\leq\frac{a}{3}+\frac{v-\frac{a}{3}}{a+2} \\
        \Longleftrightarrow\; & r\leq\frac{4a}{3}-v+1+\frac{v-\frac{a}{3}}{a+2}.
    \end{split}
\end{equation*} If $1\leq v\leq\frac{a}{3}-1$, we have that $e_2=1$ if and only if $a-v+2\leq r\leq\mathrm{min}(a+1,\frac{4a}{3}-v)=a+1$, so there are no further restrictions on $r$. Otherwise, if $\frac{a}{3}\leq v\leq a$, $e_2=1$ if and only if $a-v+2\leq r\leq\min(a+1,\frac{4a}{3}-v+1)=\frac{4a}{3}-v+1$. We conclude that $$e_2=\begin{cases}
    1 &\hbox{if }1\leq v\leq\frac{a}{3}-1\\&\hbox{or }\frac{a}{3}\leq v\leq a\hbox{ and }a-v+2\leq r\leq\frac{4a}{3}-v+1,\\
    2 &\hbox{if }\frac{a}{3}+1\leq v\leq a\hbox{ and }\frac{4a}{3}-v+2\leq r\leq a+1.
\end{cases}$$

\begin{itemize}
    \item $e_1=e_2=1$ exactly in the cases:
    \begin{itemize}
        \item $\frac{a}{3}+1\leq v\leq\frac{2a}{3}$ and $a-v+2\leq r\leq\frac{2a}{3}+1$,
        \item $\frac{2a}{3}+1\leq v\leq a$ and $a-v+2\leq r\leq\frac{4a}{3}-v+1$.
    \end{itemize} These correspond to the points $$-(2v+1)g_1-(w+v+2)g_2+(3v+4)g_3=-(v-1)-w\rho+(v+2)\rho^2-2g_3.$$
    \item $e_1=0$ and $e_2=1$ exactly in the cases:
    \begin{itemize}
        \item $1\leq v\leq\frac{a}{3}-1$,
        \item $\frac{a}{3}\leq v\leq\frac{2a}{3}-1$ and $\frac{2a}{3}+2\leq r\leq\frac{4a}{3}-v+1$.
    \end{itemize} These correspond to the points $$-(2v+2)g_1-(w+v+2)g_2+(3v+4)g_3=-v-w\rho+(v+2)\rho^2-2g_3.$$
    \item $e_1=1$ and $e_2=2$ if and only if $\frac{2a}{3}+1\leq v\leq a$ and $\frac{4a}{3}-v+2\leq r\leq\frac{2a}{3}+1$. This corresponds to the points $$-(2v+2)g_1-(w+v+3)g_2+(3v+7)g_3=-(v-1)-w\rho+(v+3)\rho^2-2g_3.$$
    \item $e_1=0$ and $e_2=2$ exactly in the cases:
    \begin{itemize}
        \item $\frac{a}{3}+1\leq v\leq\frac{2a}{3}$ and $\frac{4a}{3}-v+2\leq r\leq a+1$,
        \item $\frac{2a}{3}+1\leq v\leq a$ and $\frac{2a}{3}+2\leq r\leq a+1$.
    \end{itemize} These correspond to the points $$-(2v+3)g_1-(w+v+3)g_2+(3v+7)g_3=-v-w\rho+(v+3)\rho^2-2g_3.$$
\end{itemize}

\subsubsection*{Situation (d)}

Let us take $w=(x+1)(a+2)$ with $0\leq x\leq a-1$. In this case, $x=v+1$. First, we always have $3(r(a+1)-v)<2a(a+2)+3$, so $e_1=1$. Moreover, $3(l(a+2)-t)=3v(a+1)\leq a(a+1)-3$ if and only if $v\leq\frac{a}{3}-1$. Therefore, $e_2=1$ if $v\leq\frac{a}{3}-1$ and $e_2=2$ otherwise. Then:
\begin{itemize}
    \item $e_1=1$ and $e_2=1$ if and only if $1\leq v\leq\frac{a}{3}-1$. Together with the point $-g_2+g_3$ from \eqref{lowpoints} (so that $0\leq v\leq\frac{a}{3}-1$), this corresponds to the points $$-2vg_1-(w+v+1)g_2+(3v+1)g_3=-(v-1)-w\rho+(v+1)\rho^2-2g_3.$$
    \item $e_1=1$ and $e_2=2$ if and only if $\frac{a}{3}\leq v\leq a$. Together with the point in \eqref{highpoints} with $s=2$ (so that $\frac{a}{3}\leq v\leq a+1$), this corresponds to the points $$-(2v+1)g_1-(w+v+2)g_2+(3v+4)g_3=-(v-1)-w\rho+(v+2)\rho^2-2g_3.$$
\end{itemize}

\subsection{Representing the lattice points}\label{sec:reprpoints}

We list all lattice points in the second parallelepiped in Tables \ref{tab:points0}, \ref{tab:points1} and \ref{tab:points2} for the cases $s=0$, $s=1$ and $s=2$, respectively. We call $P_i$ (resp. $R_i$, resp. $S_i$) the set of points in the $i$-th row of Table \ref{tab:points0} (resp. \ref{tab:points1}, resp. \ref{tab:points2}).

\begin{table}
\begin{center}
    \begin{tabular}{|c|c|c|c|} \hline
        $P_i$ & Point & Situation & Region \\ \hline
        $P_1$ & \multirow{2}{*}{$-v-w\rho+(v+1)\rho^2$} & (a) & $0\leq v\leq a-1$, $1\leq r\leq a-v$ \\
        $P_2$ & & (b) & $0\leq v\leq a$, $r=a-v+1$ \\ \hline
        $P_3$ & $-v-w\rho+(v+2)\rho^2$ & (c) & $0\leq v\leq a$, $a-v+2\leq r\leq a+1$ \\ \hline
        $P_4$ & $-(v-1)-w\rho+(v+1)\rho^2$ & (d) & $0\leq v\leq a+1$, $r=0$ \\ \hline
    \end{tabular}
\caption{Lattice points in the second parallelepiped with $s=0$} \label{tab:points0}
\end{center}
\end{table}

\begin{table}
\begin{center}
    \begin{tabular}{|c|c|c|c|} \hline
        $R_i$ & Point & Situation & Region \\ \hline
        $R_1$ & \multirow{3}{*}{$-(v-1)-w\rho+(v+1)\rho^2-g_3$} & (a) & $0\leq v\leq\frac{a}{3}$, $1\leq r\leq\frac{a}{3}$ \\ 
        $R_2$ & & (a) & $\frac{a}{3}+1\leq v\leq\frac{2a}{3}-1$, $1\leq r\leq\frac{2a}{3}-v$ \\
        $R_3$ & & (d) & $0\leq v\leq\frac{2a}{3}+1$, $r=0$ \\ \hline
        $R_4$ & $-v-w\rho+(v+1)\rho^2-g_3$ & (a) & $0\leq v\leq\frac{a}{3}-1$, $\frac{a}{3}+1\leq r\leq\frac{2a}{3}-v$ \\ \hline
        $R_5$ & \multirow{5}{*}{$-(v-1)-w\rho+(v+2)\rho^2-g_3$} & (a) & $\frac{a}{3}+1\leq v\leq\frac{2a}{3}$, $\frac{2a}{3}-v+1\leq r\leq\frac{a}{3}$ \\ 
        $R_6$ & & (a) & $\frac{2a}{3}+1\leq v\leq a-1$, $1\leq r\leq a-v$ \\ 
        $R_7$ & & (b) & $\frac{2a}{3}+1\leq v\leq a$, $r=a-v+1$ \\
        $R_8$ & & (c) & $\frac{2a}{3}+1\leq v\leq a$, $a-v+2\leq r\leq\frac{a}{3}+1$ \\
        $R_9$ & & (d) & $\frac{2a}{3}+2\leq v\leq a+1$, $r=0$ \\ \hline
        $R_{10}$ & \multirow{5}{*}{$-v-w\rho+(v+2)\rho^2-g_3$} & (a) & $0\leq v\leq\frac{a}{3}$, $\frac{2a}{3}-v+1\leq r\leq a-v$ \\
        $R_{11}$ & & (a) & $\frac{a}{3}+1\leq v\leq\frac{2a}{3}-1$, $\frac{a}{3}+1\leq r\leq a-v$ \\
        $R_{12}$ & & (b) & $0\leq v\leq\frac{2a}{3}$, $r=a-v+1$ \\
        $R_{13}$ & & (c) & $1\leq v\leq\frac{2a}{3}$, $a-v+2\leq r\leq a+1$ \\
        $R_{14}$ & & (c) & $\frac{2a}{3}+1\leq v\leq a$, $\frac{a}{3}+2\leq r\leq\frac{5a}{3}-v+2$ \\ \hline
        $R_{15}$ & $-v-w\rho+(v+3)\rho^2-g_3$ & (c) & $\frac{2a}{3}+2\leq v\leq a$, $\frac{5a}{3}-v+3\leq r\leq a+1$ \\ \hline
    \end{tabular}
\caption{Lattice points in the second parallelepiped with $s=1$} \label{tab:points1}
\end{center}
\end{table}

\begin{table}
\begin{center}
    \begin{tabular}{|c|c|c|c|} \hline
        $S_i$ & Point & Situation & Region \\ \hline
        $S_1$ & \multirow{2}{*}{$-(v-1)-w\rho+(v+1)\rho^2-2g_3$} & (a) & $0\leq v\leq\frac{a}{3}-2$, $1\leq r\leq\frac{a}{3}-v-1$ \\ 
        $S_2$ & & (d) & $0\leq v\leq\frac{a}{3}-1$, $r=0$ \\ \hline
        $S_3$ & \multirow{6}{*}{$-(v-1)-w\rho+(v+2)\rho^2-2g_3$} & (a) & $0\leq v\leq\frac{a}{3}-1$, $\frac{a}{3}-v\leq r\leq\frac{2a}{3}$ \\ 
        $S_4$ & & (a) & $\frac{a}{3}\leq v\leq a-1$, $1\leq r\leq a-v$ \\ 
        $S_5$ & & (b) & $\frac{a}{3}\leq v\leq a$, $r=a-v+1$ \\ 
        $S_6$ & & (c) & $\frac{a}{3}+1\leq v\leq\frac{2a}{3}$, $a-v+2\leq r\leq\frac{2a}{3}+1$ \\
        $S_7$ & & (c) & $\frac{2a}{3}+1\leq v\leq a$, $a-v+2\leq r\leq\frac{4a}{3}-v+1$ \\
        $S_8$ & & (d) & $\frac{a}{3}\leq v\leq a+1$, $r=0$ \\ \hline
        $S_9$ & \multirow{4}{*}{$-v-w\rho+(v+2)\rho^2-2g_3$} & (a) & $0\leq v\leq\frac{a}{3}-1$, $\frac{2a}{3}+1\leq r\leq a-v$ \\
        $S_{10}$ & & (b) & $0\leq v\leq\frac{a}{3}-1$, $r=a-v+1$ \\
        $S_{11}$ & & (c) & $1\leq v\leq\frac{a}{3}-1$, $a-v+2\leq r\leq a+1$ \\
        $S_{12}$ & & (c) & $\frac{a}{3}\leq v\leq\frac{2a}{3}-1$, $\frac{2a}{3}+2\leq r\leq\frac{4a}{3}-v+1$ \\ \hline
        $S_{13}$ & $-(v-1)-w\rho+(v+3)\rho^2-2g_3$ & (c) & $\frac{2a}{3}+1\leq v\leq a$, $\frac{4a}{3}-v+2\leq r\leq\frac{2a}{3}+1$ \\ \hline
        $S_{14}$ & \multirow{2}{*}{$-v-w\rho+(v+3)\rho^2-2g_3$} & (c) & $\frac{a}{3}+1\leq v\leq\frac{2a}{3}$, $\frac{4a}{3}-v+2\leq r\leq a+1$ \\
        $S_{15}$ & & (c) & $\frac{2a}{3}+1\leq v\leq a$, $\frac{2a}{3}+2\leq r\leq a+1$ \\ \hline
    \end{tabular}
\caption{Lattice points in the second parallelepiped with $s=2$} \label{tab:points2}
\end{center}
\end{table}

For a fixed value of $s$, we identify a lattice point $\alpha_s(v,r)$ with a point $(v,r)\in\mathbb{Z}^2$. In Figures \ref{fig0}, \ref{fig1} and \ref{fig2}, we represent the previous regions in two dimensional planes, where the X-axis (resp. Y-axis) corresponds to the values of $v$ (resp. $r$). 

In the statement of Theorem \ref{teolistindec}, the element $g_3$ in (i) is the only candidate from the first parallelepiped. The elements in (ii), (iii) and (iv) come from the case $s=0$, the ones in (v) have $s=1$, and the remaining ones correspond to $s=2$. In the already mentioned figures, we highlight them in red colour.

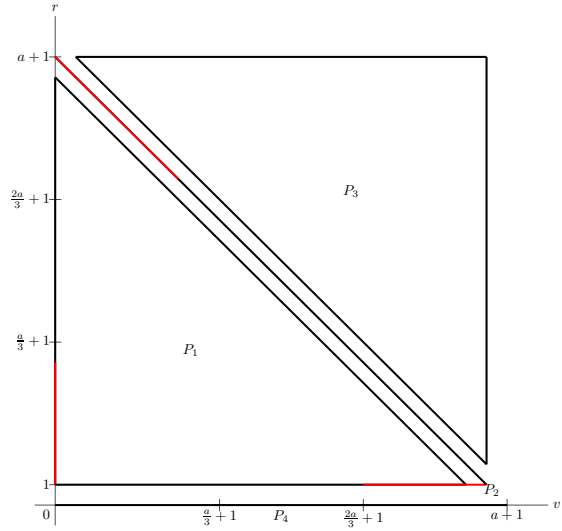
\begin{figure}
    \centering
    \resizebox{0.5\textwidth}{!}{
    \begin{tikzpicture}
        \draw[opacity=0.7] (-0.5,0) -- (12,0);
        \draw[opacity=0.7] (0,-0.5) -- (0,12);
        
        \fill (0,0) node [below left] {$0$};
        \fill (12,0) node [right] {$v$};
        \fill (0,12) node [above] {$r$};
        
        \draw[opacity=0.7] (11,-0.15) -- (11,0.15);
        \fill (11,0) node [below] {$a+1$};        
        \draw[opacity=0.7] (7.5,-0.15) -- (7.5,0.15);
        \fill (7.5,0) node [below] {$\frac{2a}{3}+1$};
        \draw[opacity=0.7] (4,-0.15) -- (4,0.15);
        \fill (4,0) node [below] {$\frac{a}{3}+1$};
        \draw[opacity=0.7] (-0.15,11) -- (0.15,11);
        \fill (0,11) node [left] {$a+1$};
        \draw[opacity=0.7] (-0.15,7.5) -- (0.15,7.5);
        \fill (0,7.5) node [left] {$\frac{2a}{3}+1$};
        \draw[opacity=0.7] (-0.15,4) -- (0.15,4);
        \fill (0,4) node [left] {$\frac{a}{3}+1$};
        \draw[opacity=0.7] (-0.15,0.5) -- (0.15,0.5);
        \fill (0,0.5) node [left] {$1$};
        
        \draw[line width=0.5mm] (0,0.5) -- (0,10.5);
        \draw[red][line width=0.5mm] (0,0.5) -- (0,3.5);
        \draw[line width=0.5mm] (0,10.5) -- (10,0.5);
        \draw[line width=0.5mm] (0,0.5) -- (10,0.5);
        \fill (3.3,3.8) node {$P_1$};
        
        \draw[red][line width=0.5mm] (7.5,0.5) -- (10.5,0.5);
        
        \draw[line width=0.5mm] (0,11) -- (10.5,0.5);
        \fill (10.625,0.375) node {$P_2$};
        
        \draw[red][line width=0.5mm] (0,11) -- (3,8);
        
        \draw[line width=0.5mm] (0.5,11) -- (10.5,11);
        \draw[line width=0.5mm] (0.5,11) -- (10.5,1);
        \draw[line width=0.5mm] (10.5,1) -- (10.5,11);
        \fill (7.2,7.7) node {$P_3$};
        
        \draw[line width=0.5mm] (0,0) -- (11,0);
        \fill (5.5,0) node [below] {$P_4$};
        
        \end{tikzpicture}
    }
    \caption{Lattice points with $s=0$}
    \label{fig0}
\end{figure}

\begin{figure}[h!]
\begin{minipage}{0.5 \textwidth}
    \centering
    \resizebox{1\textwidth}{!}{
    \begin{tikzpicture}
        \draw[opacity=0.7] (-0.5,0) -- (12,0);
        \draw[opacity=0.7] (0,-0.5) -- (0,12);
        
        \fill (0,0) node [below left] {$0$};
        \fill (12,0) node [right] {$v$};
        \fill (0,12) node [above] {$r$};
        
        \draw[opacity=0.7] (11,-0.15) -- (11,0.15);
        \fill (11,0) node [below] {$a+1$};        
        \draw[opacity=0.7] (7.5,-0.15) -- (7.5,0.15);
        \fill (7.5,0) node [below] {$\frac{2a}{3}+1$};
        \draw[opacity=0.7] (4,-0.15) -- (4,0.15);
        \fill (4,0) node [below] {$\frac{a}{3}+1$};
        \draw[opacity=0.7] (-0.15,11) -- (0.15,11);
        \fill (0,11) node [left] {$a+1$};
        \draw[opacity=0.7] (-0.15,7.5) -- (0.15,7.5);
        \fill (0,7.5) node [left] {$\frac{2a}{3}+1$};
        \draw[opacity=0.7] (-0.15,4) -- (0.15,4);
        \fill (0,4) node [left] {$\frac{a}{3}+1$};
        \draw[opacity=0.7] (-0.15,0.5) -- (0.15,0.5);
        \fill (0,0.5) node [left] {$1$};
        
        \draw[line width=0.5mm] (0,3.5) -- (3.5,3.5);
        \draw[line width=0.5mm] (0,0.5) -- (0,3.5);
        \draw[line width=0.5mm] (0,0.5) -- (3.5,0.5);
        \draw[line width=0.5mm] (3.5,3.5) -- (3.5,0.5);
        \fill (1.75,2) node {$R_1$};
        
        \draw[line width=0.5mm] (4,0.5) -- (6.5,0.5);
        \draw[line width=0.5mm] (4,0.5) -- (4,3);
        \draw[line width=0.5mm] (4,3) -- (6.5,0.5);
        \fill (4.8,1.3) node {$R_2$};
        
        \draw[line width=0.5mm] (0,0) -- (7.5,0);
        \fill (5.5,0) node [below] {$R_{3}$};

        \fill[red] (0,4) -- (3,4) -- (0,7) -- cycle;
        \draw[line width=0.5mm] (0,4) -- (3,4);
        \draw[line width=0.5mm] (0,4) -- (0,7);
        \draw[line width=0.5mm] (0,7) -- (3,4);
        \fill (1,5) node {$R_4$};
        
        \draw[line width=0.5mm] (4,3.5) -- (7,3.5);
        \draw[line width=0.5mm] (4,3.5) -- (7,0.5);
        \draw[line width=0.5mm] (7,0.5) -- (7,3.5);
        \fill (6,2.5) node {$R_5$};
        
        \draw[line width=0.5mm] (7.5,0.5) -- (10,0.5);
        \draw[line width=0.5mm] (7.5,0.5) -- (7.5,3);
        \draw[line width=0.5mm] (10,0.5) -- (7.5,3);
        \fill (8.3,1.3) node {$R_6$};
        
        \draw[line width=0.5mm] (10.5,0.5) -- (7.5,3.5);
        \fill (8.75,2) node {$R_7$};

        \draw[line width=0.5mm] (7.5,4) -- (10.5,4);
        \draw[line width=0.5mm] (7.5,4) -- (10.5,1);
        \draw[line width=0.5mm] (10.5,4) -- (10.5,1);
        \fill (9.5,3) node {$R_8$};
        
        \draw[line width=0.5mm] (8,0) -- (11,0);
        \fill (9.5,0) node [below] {$R_9$};
        
        \draw[line width=0.5mm] (0,7.5) -- (0,10.5);
        \draw[line width=0.5mm] (0,7.5) -- (3.5,4);
        \draw[line width=0.5mm] (0,10.5) -- (3.5,7);
        \draw[line width=0.5mm] (3.5,4) -- (3.5,7);
        \fill (1.75,7.25) node {$R_{10}$};
        
        \draw[line width=0.5mm] (4,4) -- (6.5,4);
        \draw[line width=0.5mm] (4,4) -- (4,6.5);
        \draw[line width=0.5mm] (4,6.5) -- (6.5,4);
        \fill (4.8,4.8) node {$R_{11}$};
        
        \draw[line width=0.5mm] (0,11) -- (7,4);
        \fill (3.75,7.25) node [below] {$R_{12}$};
        
        \draw[line width=0.5mm] (0.5,11) -- (7,11);
        \draw[line width=0.5mm] (0.5,11) -- (7,4.5);
        \draw[line width=0.5mm] (7,11) -- (7,4.5);
        \fill (4.9,8.9) node {$R_{13}$};
        
        \draw[line width=0.5mm] (7.5,4.5) -- (10.5,4.5);
        \draw[line width=0.5mm] (7.5,4.5) -- (7.5,11);
        \draw[line width=0.5mm] (7.5,11) -- (10.5,8);
        \draw[line width=0.5mm] (10.5,4.5) -- (10.5,8);
        \fill (9,7.75) node {$R_{14}$};
        
        \draw[line width=0.5mm] (8,11) -- (10.5,11);
        \draw[line width=0.5mm] (8,11) -- (10.5,8.5);
        \draw[line width=0.5mm] (10.5,11) -- (10.5,8.5);
        \fill (9.7,10.2) node {$R_{15}$};
        
        \end{tikzpicture}
    }
    \caption{Lattice points with $s=1$}
    \label{fig1}
\end{minipage}
\begin{minipage}{0.5 \textwidth}
    \centering
    \resizebox{1\textwidth}{!}{
    \begin{tikzpicture}
        \draw[opacity=0.7] (-0.5,0) -- (12,0);
        \draw[opacity=0.7] (0,-0.5) -- (0,12);
        
        \fill (0,0) node [below left] {$0$};
        \fill (12,0) node [right] {$v$};
        \fill (0,12) node [above] {$r$};
        
        \draw[opacity=0.7] (11,-0.15) -- (11,0.15);
        \fill (11,0) node [below] {$a+1$};        
        \draw[opacity=0.7] (7.5,-0.15) -- (7.5,0.15);
        \fill (7.5,0) node [below] {$\frac{2a}{3}+1$};
        \draw[opacity=0.7] (4,-0.15) -- (4,0.15);
        \fill (4,0) node [below] {$\frac{a}{3}+1$};
        \draw[opacity=0.7] (-0.15,11) -- (0.15,11);
        \fill (0,11) node [left] {$a+1$};
        \draw[opacity=0.7] (-0.15,7.5) -- (0.15,7.5);
        \fill (0,7.5) node [left] {$\frac{2a}{3}+1$};
        \draw[opacity=0.7] (-0.15,4) -- (0.15,4);
        \fill (0,4) node [left] {$\frac{a}{3}+1$};
        \draw[opacity=0.7] (-0.15,0.5) -- (0.15,0.5);
        \fill (0,0.5) node [left] {$1$};
        
        \draw[line width=0.5mm] (0,0.5) -- (2.5,0.5);
        \draw[line width=0.5mm] (0,0.5) -- (0,3);
        \draw[line width=0.5mm] (0,3) -- (2.5,0.5);
        \fill (0.8,1.3) node {$S_1$};
        
        \draw[line width=0.5mm] (0,0) -- (3,0);
        \fill (1.5,0) node [below] {$S_2$};
        
        \draw[red][line width=0.5mm] (0,0) -- (0,3);

        \draw[line width=0.5mm] (0,3.5) -- (3,0.5);
        \draw[line width=0.5mm] (0,3.5) -- (0,7);
        \draw[line width=0.5mm] (0,7) -- (3,7);
        \draw[line width=0.5mm] (3,7) -- (3,0.5);
        \fill (1.5,3.7) node {$S_3$};
        
        \draw[line width=0.5mm] (3.5,0.5) -- (10,0.5);
        \draw[line width=0.5mm] (3.5,0.5) -- (3.5,7);
        \draw[line width=0.5mm] (10,0.5) -- (3.5,7);
        \fill (5.7,3) node {$S_4$};
        
        \draw[line width=0.5mm] (10.5,0.5) -- (3.5,7.5);
        \fill (7.125,4.125) node {$S_5$};
        
        \draw[line width=0.5mm] (4,7.5) -- (7,7.5);
        \draw[line width=0.5mm] (4,7.5) -- (7,4.5);
        \draw[line width=0.5mm] (7,4.5) -- (7,7.5);
        \fill (6,6.5) node {$S_6$};
        
        \draw[line width=0.5mm] (7.5,4) -- (7.5,7);
        \draw[line width=0.5mm] (7.5,4) -- (10.5,1);
        \draw[line width=0.5mm] (7.5,7) -- (10.5,4);
        \draw[line width=0.5mm] (10.5,1) -- (10.5,4);
        \fill (9.2,4) node {$S_7$};
        
        \draw[line width=0.5mm] (3.5,0) -- (11,0);
        \fill (8.75,0) node [below] {$S_8$};
        
        \draw[red][line width=0.5mm] (0,7.5) -- (3,7.5);
        \draw[line width=0.5mm] (0,7.5) -- (0,10.5);
        \draw[line width=0.5mm] (0,10.5) -- (3,7.5);
        \fill (1,8.5) node {$S_9$};
        
        \draw[line width=0.5mm] (0,11) -- (3,8);
        \fill (3.125,7.875) node {$S_{10}$};
        
        \draw[line width=0.5mm] (0.5,11) -- (3,11);
        \draw[line width=0.5mm] (0.5,11) -- (3,8.5);
        \draw[line width=0.5mm] (3,8.5) -- (3,11);
        \fill (2.2,10.2) node {$S_{11}$};
        
        \draw[line width=0.5mm] (3.5,8) -- (3.5,11);
        \draw[line width=0.5mm] (3.5,8) -- (6.5,8);
        \draw[red][line width=0.5mm] (6.5,8) -- (3.5,11);
        \fill (4.5,9) node {$S_{12}$};
        
        \draw[line width=0.5mm] (7.5,7.5) -- (10.5,7.5);
        \draw[line width=0.5mm] (7.5,7.5) -- (10.5,4.5);
        \draw[line width=0.5mm] (10.5,4.5) -- (10.5,7.5);
        \fill (9.5,6.5) node {$S_{13}$};
        
        \draw[line width=0.5mm] (4,11) -- (7,11);
        \draw[line width=0.5mm] (4,11) -- (7,8);
        \draw[line width=0.5mm] (7,8) -- (7,11);
        \fill (6,10) node {$S_{14}$};
        
        \draw[line width=0.5mm] (7.5,8) -- (10.5,8);
        \draw[line width=0.5mm] (7.5,8) -- (7.5,11);
        \draw[line width=0.5mm] (7.5,11) -- (10.5,11);
        \draw[line width=0.5mm] (10.5,8) -- (10.5,11);
        \fill (9,9.5) node {$S_{15}$};
        
        \end{tikzpicture}
    }
    \caption{Lattice points with $s=2$}
    \label{fig2}
\end{minipage}
\end{figure}

\subsection{Proof of the main theorem}

We proceed to validate Theorem \ref{teolistindec}. We need to prove that the elements listed therein are indecomposable and that all the other lattice points in the parallelepipeds are decomposable.

\subsubsection{Proving the indecomposability} \label{subsubsec:indecomposable}

We prove the indecomposability of the elements in the statement of Theorem \ref{teolistindec}. The main tool will be the codifferent $\mathcal{O}_K^{\vee}$ of $K$, for which we first determine an $\mathcal{O}_K$-basis $\{\varphi_j\}_{j=1}^3$. Note that since the codifferent is the dual lattice of $\mathcal{O}_K$, its Gram matrix is the inverse of the Gram matrix of an integral basis.

Let $g_1=1$, $g_2=\rho$, $g_3=\frac{1+\rho+\rho^2}{3}$, which as already mentioned is an integral basis of $K$. Let $M=(\mathrm{Tr}(g_ig_k))_{i,k=1}^3$. Then the coordinates of $\varphi_j$ with respect to the basis $\{g_i\}_{i=1}^3$ is given by the $j$-th column of the matrix $M^{-1}$. 
Carrying out the product traces, we have 
$$M=\begin{pmatrix}
3 & a & \frac{a^2+3a+9}{3} \\
a & a^2+2a+6 & \frac{(a+1)(a^2+3a+9)}{3} \\
\frac{a^2+3a+9}{3} & \frac{(a+1)(a^2+3a+9)}{3} & \frac{(a^2+3a+5)(a^2+3a+9)}{9}
\end{pmatrix}.$$ From this, we find $$M^{-1}=\frac{1}{a^2+3a+9}\begin{pmatrix}
a^2+7a+21 & a^2+7a+9 & -3(a+6) \\
a^2+7a+9 & 2(a^2+3a+3) & -3(2a+3) \\
-3(a+6) & -3(2a+3) & 18
\end{pmatrix}.$$ Then, we conclude that \begin{equation*}
    \begin{split}
        \varphi_1&=\frac{1}{a^2+3a+9}\Big(a^2+7a+21+(a^2+7a+9)\rho-3(a+6)\frac{1+\rho+\rho^2}{3}\Big), \\ \varphi_2&=\frac{1}{a^2+3a+9}\Big(a^2+7a+9+2(a^2+3a+3)\rho-3(2a+3)\frac{1+\rho+\rho^2}{3}\Big), \\ \varphi_3&=\frac{1}{a^2+3a+9}\Big(-3(a+6)-3(2a+3)\rho+18\frac{1+\rho+\rho^2}{3}\Big).
    \end{split}
\end{equation*}

We begin with the elements of minimal trace 1.

\begin{pro}\label{proindecs12} The elements in Theorem \ref{teolistindec} (i), (v), (vi), (vii) and (viii) have minimal trace $1$, and consequently they are indecomposables.
\end{pro}
\begin{proof}
For (i), i.e. the point $g_3$, we just note that $\varphi_1+\varphi_3\in\mathcal{O}_K^{\vee}$ is a root of the polynomial $$x^3-x^2+\frac{a+3}{a^2+3a+9}x-\frac{2a+3}{(a^2+3a+9)^2},$$ so it is totally positive, and that $$\mathrm{Tr}((\varphi_1+\varphi_3)g_3)=1.$$

For (v), let $0\leq v\leq\frac{a}{3}-1$ and $\frac{a}{3}+1\leq r\leq\frac{2a}{3}-v$. The element $3\varphi_1+2\varphi_3\in\mathcal{O}_K^{\vee}$ is a root of the polynomial $$x^3-3x^2+\frac{18}{a^2+3a+9}x-\frac{27}{(a^2+3a+9)^2},$$ so it is totally positive. Now, we have \begin{equation*}
    \begin{split}
        &\mathrm{Tr}((3\varphi_1+2\varphi_3)(-(2v+1)g_1-(v(a+3)+r+1)g_2+(3v+2)g_3)=\\&-6v-3+6v+4=1.
    \end{split}
\end{equation*}

Since the elements in (vii) and (viii) are associated with conjugates of elements in (vi), it is enough to work with the last ones. Let $0\leq r\leq\frac{a}{3}-1$. We already know that $\varphi_1+\varphi_3\in\mathcal{O}_K^{\vee,+}$. Now, we have $$\mathrm{Tr}((\varphi_1+\varphi_3)(-(r+1)g_2+g_3))=1.$$
\end{proof}

On the other hand, the remaining indecomposables have minimal trace $2$.

\begin{pro}\label{promintrace2} The elements in Theorem \ref{teolistindec} (ii), (iii) and (iv) have minimal trace $2$.
\end{pro}
\begin{proof}
Let us recover the notation in Section \ref{sects=0}. We have that $T_1$ maps bijectively the points from (ii) to (iv), from (iv) to (iii), and from (iii) to (ii). This means that the elements in (iii) and (iv) are associated with conjugates of elements in (ii). Moreover, the points in (ii) lie in $\blacktriangle_0(a)$. Therefore, it is enough to prove the statement of the points in (ii).

Let $\alpha=-g_1-(r+1)g_2+3g_3$, $1\leq r\leq\frac{a}{3}$, be a lattice point as in (ii). Since $\mathrm{Tr}((\varphi_1+\varphi_3)\alpha)=2$, the minimal trace of $\alpha$ is upper bounded by $2$. Let us assume that there exists a totally positive element of the codifferent $\delta=u_1\varphi_1+u_2\varphi_2+u_3\varphi_3$ such that $\mathrm{Tr}(\alpha\delta)=1$, i.e., $$\mathrm{Tr}(\alpha\delta)=-u_1-(r+1)u_2+3u_3=1.$$ It implies that $u_1=-(r+1)u_2+3u_3-1$. Now we will consider several totally positive elements, whose product with $\delta$ gives a totally positive element with a positive trace. We will show that these traces cannot be all positive at the same time.

First of all, let us consider the element $-g_1-rg_2+3g_3$. This element is totally positive for $1\leq r\leq\frac{a}{3}$. Note that it is also true for $r=1$ as $-g_1-g_2+3g_3=\rho^2$. Thus, we obtain $$\mathrm{Tr}(\delta(-g_1-rg_2+3g_3))=(r+1)u_2-3u_3+1-ru_2+3u_3=u_2+1.$$ This trace is positive only if $u_2\geq0$. Similarly, when we consider the totally positive element $-g_1-(r+2)g_2+3g_3$, we get $$\mathrm{Tr}(\delta(-g_1-(r+2)g_2+3g_3))=(r+1)u_2-3u_3+1-(r+2)u_2+3u_3=-u_2+1.$$ This trace is positive only if $u_2\leq0$, which, together with the previous part, gives $u_2=0$.

When we consider the element $g_3$, we can conclude that $$\mathrm{Tr}(\delta g_3)=u_3,$$ which implies $u_3>0$. On the other hand, for $-g_1-(\frac{a}{3}+2)g_2+2g_3$, we see that $$\mathrm{Tr}\Big(\delta\Big(-g_1-\Big(\frac{a}{3}+2\Big)g_2+2g_3\Big)\Big)=-3u_3+1+2u_3=-u_3+1,$$ which leads to $u_3\leq0$. Thus, there is no $u_3\in\mathbb{Z}$ for which $\delta$ would be totally positive.
\end{proof}

Next, we prove that these are indecomposable.

\begin{pro}\label{proindecs0} The elements in Theorem \ref{teolistindec} (ii), (iii) and (iv) are indecomposable.
\end{pro}
\begin{proof}
As in Proposition \ref{promintrace2}, it is enough to prove the indecomposability for points in (ii).

Let $\alpha=-g_1-(r+1)g_2+3g_3$, $1\leq r\leq\frac{a}{3}$, be a lattice point as in (ii). Since it has minimal trace $2$, its only possible decomposition is as a sum of elements whose trace after multiplication by $\varphi_1+\varphi_3$ is equal to $1$, i.e. $$\alpha=(-v_1g_1-w_1g_2+(v_1+1)g_3)+(-v_2g_1-w_2g_2+(v_2+1)g_3).$$ First of all, let us discuss the case when $v_1<0$, and firstly assume $v_1\leq-2$. If $w_1>0$, then $$-v_1g_1-w_1g_2+(v_1+1)g_3<0,$$ since $a+1<\rho<a+1+\frac{2}{a}$ and $g_3\geq3$ for $a\geq 21$. Thus, in this case we must have $w_1\leq0$. Moreover, we can conclude that \begin{equation*}
    \begin{split}
        \mathrm{Tr}(-v_1g_1-w_1g_2+(v_1+1)g_3)&=-3v_1-aw_1+(v_1+1)\frac{a^2+3a+9}{3}\\&=\frac{a^2+3a}{3}v_1-aw_1+\frac{a^2+3a+9}{3}.
    \end{split}
\end{equation*} This trace is negative or zero if $$\frac{(a+3)(v_1+1)}{3}+\frac{3}{a}\leq w_1,$$ in which case our element is not totally positive. Thus, let us assume that $\frac{(a+3)(v_1+1)}{3}+\frac{3}{a}>w_1$. We have $$-v_1g_1-w_1g_2+(v_1+1)g_3=\frac{-2v_1+1}{3}-\Big(w_1-\frac{v_1+1}{3}\Big)\rho+\frac{v_1+1}{3}\rho^2.$$ Under our assumption on $v_1$ and $w_1$, we see that the coefficient before $\rho$ is positive, and the coefficient before $\rho^2$ is negative. Since $-1-\frac{1}{a}<\rho'<-1-\frac{1}{2a}$, we can conclude that \begin{equation*}
    \begin{split}
        \frac{-2v_1+1}{3}-\Big(w_1-\frac{v_1+1}{3}\Big)\rho'+\frac{v_1+1}{3}\rho'^2<\frac{-2v_1+1}{3}+\Big(w_1-\frac{v_1+1}{3}\Big)+\frac{v_1+1}{3}\\\leq\frac{-2v_1+1}{3}+\frac{(a+3)(v_1+1)}{3}+\frac{3}{a}-\frac{v_1+1}{3}+\frac{v_1+1}{3}=\frac{(a+1)(v_1+1)}{3}+1+\frac{3}{a}\leq0
    \end{split}
\end{equation*} for $a\geq21$. It implies that if $v_1\leq-2$, then the element of this form is never totally positive.

If $v_1=-1$, we obtain the element $1-w_1\rho$. If $w_1>0$, then $1-w_1\rho<0$, and if $w_1<0$, then $1-w_1\rho'<0$. We obtain a totally positive element only if $w_1=0$, i.e., the element $1$. However, we know that $\alpha$ is smaller than $1$ for some embedding, so $1$ cannot appear in its decomposition.

Therefore, we must have $v_1\geq0$, and this is also true for $v_2$. Without loss of generality, we thus must have $v_1=1$ and $v_2=0$. If $w_2<0$, we get $$-w_2\rho'+\frac{1+\rho'+\rho'^2}{3}=\frac{1}{3}-\Big(w_2-\frac{1}{3}\Big)\rho'+\frac{1}{3}\rho'^2<\frac{1}{3}+w_2-\frac{1}{3}+\frac{\Big(1+\frac{1}{a}\Big)^2}{3}=w_2+\frac{\Big(1+\frac{1}{a}\Big)^2}{3}<0$$ as, clearly, $\frac{\Big(1+\frac{1}{a}\Big)^2}{3}<1$ for $a\geq21$. Therefore, $w_2\geq0$.

Similarly, if $w_1<0$, we obtain $$-1+\frac{2}{3}-\Big(w_1-\frac{2}{3}\Big)\rho'+\frac{2}{3}\rho'^2<-1+\frac{2}{3}+w_1-\frac{2}{3}+\frac{2}{3}\Big(1+\frac{1}{a}\Big)^2<0.$$ Thus, we have $0\leq w_1,w_2\leq r+1\leq\frac{a}{3}+1$.

Now, we will use the remaining root $-\frac{1}{a+2}<\rho''<-\frac{1}{a+3}$. If $w_1=0$, we have $$-1+2\frac{1+\rho''+\rho''^2}{3}<-\frac{1}{3}-\frac{2}{3}\frac{1}{a+3}+\frac{2}{3}\frac{1}{(a+2)^2}<0$$ for $a\geq 21$. Similarly, if $w_1>0$, we obtain
\[
-1-w_1\rho''+2\frac{1+\rho''+\rho''^2}{3}=-\frac{1}{3}-\left(w_1-\frac{2}{3}\right)\rho''+\frac{2}{3}\rho''<-\frac{1}{3}+\left(w_1-\frac{2}{3}\right)\frac{1}{a+2}+\frac{2}{3}\frac{1}{(a+2)^2}
\]
The expression on the right side is positive only if $$w_1>\frac{a}{3}+\frac{4}{3}-\frac{2}{3(a+2)}>\frac{a}{3}+1$$ for $a\geq21$. This is impossible as we must have $w_1\leq\frac{a}{3}+1$. Therefore, $\alpha$ is indecomposable.
\end{proof}

\subsubsection{Transformations of lattice points}\label{sect:t1t2}

For later use, we consider the transformations $T_1$ and $T_2$ introduced in \cite[Section 5.1]{kalatinkova} and extend them to all elements of $K$. Namely, for $\alpha\in K$, we define $$T_1(\alpha)=\alpha'(-1-a-(a^2+3a+3)\rho+(a+2)\rho^2),\quad T_2(\alpha)=\alpha''\rho^2.$$

The choice $s=0$ corresponds to the ones studied in \cite{kalatinkova}. Our points $\alpha_0(v,r)$ are represented therein as $\alpha(v,W)$ with $W=r-1$. With our notation, the definitions of $T_1$ and $T_2$ in $\blacktriangle$ are:
\begin{equation*}
    \begin{split}
        &T_1(\alpha_0(v,r))=(\alpha_0(v,r))'(-1-a-(a^2+3a+3)\rho+(a+2)\rho^2)=\alpha_0(r-1,a+2-v-r),\\&T_2(\alpha_0(v,r))=(\alpha_0(v,r))''\rho^2=\alpha_0(a+1-v-r,v+1).
    \end{split}
\end{equation*}

Let us study the behaviour of transformations $T_1$ and $T_2$ over lattice points with $s\neq0$. It turns out that their behaviour over these points is far more unpredictable than in the case $s=0$. For the sake of simplicity, we will just work with the transformation $T_1$.

Let us assume that $s=1$ and let us identify a lattice point $\alpha_1(v,r)$ with a point $(v,r)\in\mathbb{Z}^2$. Using some mathematical software, we can easily check the following:
\begin{itemize}
    \item If $(v,r)\in\bigcup_{i=1}^3R_i$, $$T_1(v,r)=\begin{cases}
    \Big(\frac{2a}{3}+1,0\Big) & \hbox{if }(v,r)=(0,0),\\
    \Big(r+\frac{2a}{3},a+2-v-r\Big) & \hbox{otherwise}.
    \end{cases}
    $$
    \item If $(v,r)\in R_4$, $$T_1(v,r)=\Big(r-\frac{a}{3}-1,a+1-v-r\Big).$$
    \item If $(v,r)\in\bigcup_{i=5}^9R_i$, $$T_1(v,r)=\begin{cases}
    \Big(r+\frac{2a}{3},a+2-v-r\Big) & \hbox{if }v+r\leq a+2, \\
    \Big(r+\frac{2a}{3}-1,2a+4-v-r\Big) & \hbox{if }v+r>a+2.
    \end{cases}$$
    \item If $(v,r)\in\bigcup_{i=10}^{14}R_i$, $$T_1(v,r)=\begin{cases}
    \Big(r-\frac{a}{3}-1,a+1-v-r\Big) & \hbox{if }v+r\leq a+1, \\
    \Big(r-\frac{a}{3}-2,2a+3-v-r\Big) & \hbox{if }v+r>a+1.
    \end{cases}$$
    \item If $(v,r)\in R_{15}$, $$T_1(v,r)=\Big(r-\frac{a}{3}-2,2a+3-v-r\Big).$$
\end{itemize}

We see that at each case $T_1$ behaves as a composition of symmetries and translations. In the following, we list the images of some regions $R_i$ by $T_1$.

\begin{lema}\label{lemat1s=1} Under the previous notations, we have:
\begin{itemize}
    \item $T_1(R_1\cup R_2)=R_{14}$,
    \item $T_1(R_3-\{(0,0)\})=\{v=\frac{2a}{3},\,\frac{a}{3}+1\leq r\leq a+1\}\subset R_{12}\cup R_{13}$,
    \item $T_1(0,0)=(\frac{2a}{3}+1,0)$,
    \item $T_1(R_{10}\cup R_{11})=R_1\cup R_2$,
    \item $T_1(R_{12})=R_3-\{(\frac{2a}{3}+1,0)\}$,
    \item $T_1(R_5)=R_8$,
    \item $T_1(R_9\cup R_{15})=R_5$.
\end{itemize}
\end{lema}
\begin{proof}
We may check directly each inequality. Let us do it explicitly for the first one. Given $(v,r)\in R_1\cup R_2$, we know that $T_1(v,r)=(v',r')$, where $$\begin{cases}v'=r+\frac{2a}{3},\\ r'=a+2-v-r.\end{cases}$$ If $(v,r)\in R_1$, we have that $0\leq v\leq\frac{a}{3}$ and $1\leq r\leq\frac{a}{3}$. We see directly that $\frac{2a}{3}+1\leq v'\leq a$. On the other hand, $\frac{2a}{3}+2-r\leq r'\leq a+2-r$. Now, $\frac{2a}{3}+2-r=\frac{4a}{3}+2-v'$ and $a+2-r=\frac{5a}{3}-v'+2$, so $\frac{4a}{3}-v'+2\leq r'\leq\frac{5a}{3}-v'+2$. If $(v,r)\in R_2$, we have that $\frac{a}{3}+1\leq v\leq\frac{2a}{3}-1$ and $1\leq r\leq \frac{2a}{3}-v$. It is direct that $\frac{2a}{3}+1\leq v'\leq\frac{4a}{3}-v+1\leq a$. On the other hand, since $v+r\leq\frac{2a}{3}$, we have that $r'\geq\frac{2a}{3}+2$. Finally, $v'+r'=\frac{5a}{3}+2-v\leq\frac{4a}{3}+1$, so $r'\leq\frac{4a}{3}-v'+1$. Joining the two cases, we see that $\frac{2a}{3}+1\leq v'\leq a$ and $\frac{2a}{3}+2\leq r'\leq\frac{5a}{3}-v'+2$. Since all these steps are reversible, we conclude that $T_1(R_1\cup R_2)=R_{14}$.
\end{proof}

Now, let us assume that $s=2$ and study the behaviour of $T_1$ over the points $\alpha_2(v,r)$. As in the previous case, we identify a lattice point $\alpha_2(v,r)$ with a point $(v,r)\in\mathbb{Z}^2$.

\begin{itemize}
    \item If $(v,r)\in S_1\cup S_2$, $$T_1(v,r)=\Big(r+\frac{a}{3},a+1-r-v\Big).$$
    \item If $(v,r)\in S_3\cup S_4\cup S_5\cup S_6\cup S_7\cup S_8$, $$T_1(v,r)=\begin{cases}
    \Big(r+\frac{a}{3},a+1-v-r\Big) & \hbox{if }v+r\leq a+1, \\
    \Big(r+\frac{a}{3}-1,2a+3-v-r\Big) & \hbox{if }v+r>a+1.
    \end{cases}$$
    \item If $(v,r)\in S_9\cup S_{10}\cup S_{11}\cup S_{12}$, $$T_1(v,r)=\begin{cases}
    \Big(r-\frac{2a}{3}-1,a-r-v\Big) & \hbox{if }v+r\leq a, \\
    \Big(r-\frac{2a}{3}-2,2a-r-v+2\Big) & \hbox{if }v+r> a.
    \end{cases}$$
    \item If $(v,r)\in S_{13}$, $$T_1(v,r)=(r+\frac{a}{3}-1,2a+3-v-r).$$
    \item If $(v,r)\in S_{14}\cup S_{15}$, $$T_1(v,r)=\Big(r-\frac{2a}{3}-2,2a+2-v-r\Big).$$
\end{itemize}

Again, we can use these expressions to identify the images of regions $S_i$ by $T_1$.

\begin{lema}\label{lemat1s=2} Let us identify a lattice point $\alpha_2(v,r)$ with a point $(v,r)\in\mathbb{Z}^2$. We have:
\begin{itemize}
    \item $T_1(S_1\cup S_2)=S_{12}$,
    \item $T_1(S_3)=S_6\cup S_7$,
    \item $T_1(S_6\cup S_7)=S_{14}\cup S_{15}$,
    \item $T_1(S_{12})=S_9$.
\end{itemize}
\end{lema}

\subsubsection{Finding decompositions of lattice points}\label{sec:decspoints}

We have proved that the elements in Theorem \ref{teolistindec} are indecomposable, and we have given their minimal traces. Now, we prove that all the other lattice points in the second parallelepiped are decomposable. Again, we have to distinguish different cases depending on the value of $s$.

Let us assume that $s=0$. We already know that the points from $P_3$ and $P_4$ are decomposable (because they are indecomposables in $\mathbb{Z}[\rho]$, see \cite[Lemma 4.2]{kalatinkova}). Thus, we are left with the regions $P_1$ and $P_2$. Their union correspond to the triangle $\blacktriangle$ (note that since the points in situation (d) are discarded, we can denote $v=x$).

Recall the definition of $T_1$ and $T_2$ in Section \ref{sect:t1t2}. Since the factors multiplying the conjugate of $\alpha_0(v,r)$ at each line are totally positive units, we have that for each $i\in\{1,2\}$, $\alpha_0(v,r)$ is indecomposable if and only if so is $T_i(\alpha_0(v,r))$. Moreover, for each $\alpha\in\blacktriangle$, there is $i\in\{1,2\}$ such that $T_i(\alpha)$ lies in the region $$\blacktriangle_0(a)=\Big\{\alpha_0(v,r)\,|\,0\leq v\leq\frac{a}{3}-1,\,v+1\leq r\leq a-2v\Big\}\cup\Big\{\alpha_0\Big(\frac{a}{3},\frac{a}{3}+1\Big)\Big\}.$$ Therefore, it is enough to study the decomposability of the lattice points in $\blacktriangle_0(a)$.

\begin{pro}\label{pro:decs0} The elements with $s=0$ other than the ones in Theorem \ref{teolistindec} (ii), (iii) and (iv) are decomposable.
\end{pro}
\begin{proof}
We know that it is enough to look at the points in $\blacktriangle_0(a)$. Thus, let us consider $\alpha\in\blacktriangle_0(a)$, i.e. we have \begin{equation}\label{expralpha}
    \begin{split}
        \alpha=-v-(v(a+2)+r)\rho+(v+1)\rho^2,
    \end{split}
\end{equation} where $0\leq v\leq\frac{a}{3}-1$ and $v+1\leq r\leq a-2v$, or $\alpha=\alpha_0(\frac{a}{3},\frac{a}{3}+1)$. Among these, the points that are known to be indecomposable (from Proposition \ref{proindecs0}) are the ones in Theorem \ref{teolistindec} (ii), i.e. for $v=0$ and $1\leq r\leq \frac{a}{3}$.

Let us assume that $0\leq v\leq\frac{a}{3}-1$ and $\frac{a}{3}+1\leq r\leq\frac{2a}{3}-v$. Then $\alpha-g_3$ is one of the points in Theorem \ref{teolistindec} (v), so $\alpha$ is decomposable.

If $v=0$ and $\frac{2a}{3}+1\leq r\leq a$, we have that \begin{equation*}
    \begin{split}
        \alpha&=\Big[-\frac{2a}{3}\rho+\rho^2-g_3\Big]\\&+\Big[1-\Big(r-\frac{2a}{3}-1\Big)\rho+\rho^2-2g_3\Big].
    \end{split}
\end{equation*} The first summand is $\alpha_1(0,\frac{2a}{3})\in R_4$. On the other hand, $0\leq r'\leq\frac{a}{3}-1$ for $r'=r-\frac{2a}{3}-1$, so the second summand is the point $\alpha_2(0,r')$ that belongs to $S_1$ if $1\leq r'\leq\frac{a}{3}-1$ and to $S_2$ if $r'=0$. Therefore, $\alpha$ is decomposable for this case.

Assume that $1\leq v\leq\frac{a}{3}-1$ and $v+1\leq r\leq\frac{a}{3}$. We have \begin{equation*}
    \begin{split}
        \alpha&=\Big[-\Big(\frac{a}{3}+1\Big)\rho+\rho^2-g_3\Big]\\&+\Big[-(v-1)-(v(a+2)+r-\frac{a}{3}-2)\rho+(v+1)\rho^2-2g_3\Big].
    \end{split}
\end{equation*} The first summand is $\alpha_1(0,\frac{a}{3}+1)\in R_4$. The second one can be rewritten as $$-(v-1)-\Big((v-1)(a+2)+r+\frac{2a}{3}\Big)\rho+(v+1)\rho^2-2g_3.$$ For $v'=v-1$ and $r'=r+\frac{2a}{3}$, this is of the form $$-v'-(v'(a+2)+r')\rho+(v'+2)\rho^2-2g_3,$$ with $0\leq v'\leq\frac{a}{3}-2$.
\begin{itemize}
    \item If $r\leq\frac{a}{3}-v+1$, then $\frac{2a}{3}+1+v\leq r'\leq a-v+1=a-v'$ and the point above is $\alpha_2(v',r')\in S_9$.
    \item If $r=\frac{a}{3}-v+2$, then $r'=a-v'+1$ and we obtain $\alpha_2(v',r')\in S_{10}$.
    \item In the remaining cases, $a-v'+2\leq r'\leq a$ and we obtain $\alpha_2(v',r')\in S_{11}$.
\end{itemize}
Now, consider the same decomposition for $\alpha=\alpha_0(\frac{a}{3},\frac{a}{3}+1)$, i.e. $v=\frac{a}{3}$ and $r=\frac{a}{3}+1$. Then $v'=\frac{a}{3}-1$ and $r'=a+1$, so that $\alpha_2(v',r')\in S_{11}$.

Assume that $1\leq v\leq\frac{a}{3}-1$ and $\frac{2a}{3}-v+1\leq r\leq a-2v$. We have \begin{equation*}
    \begin{split}
        \alpha&=\Big[-\frac{2a}{3}\rho+\rho^2-g_3\Big]\\&+\Big[-(v-1)-\Big(v(a+2)+r-\frac{2a}{3}-1\Big)\rho+(v+1)\rho^2-2g_3\Big].
    \end{split}
\end{equation*} We already know that the first summand is totally positive. Let us assume that $\frac{2a}{3}-v+1\leq r\leq\frac{2a}{3}$ (for simplicity, even when $a-2v\leq\frac{2a}{3}$, which happens if and only if $v\geq\lceil\frac{a}{6}\rceil$ and in which case $\alpha$ is as in \eqref{expralpha}). We rewrite the second summand as $$-(v-1)-\Big((v-1)(a+2)+r+\frac{a}{3}+1\Big)\rho+(v+1)\rho^2-2g_3,$$ which is of the form $$-v'-(v'(a+2)+r')\rho+(v'+2)\rho^2-2g_3$$ for $v'=v-1$ and $r'=r+\frac{a}{3}+1$, so that $a-v'+1\leq r'\leq a+1$. If $v=1$, then $v'=0$ and $r'=a+1$, and we obtain the lattice point $\alpha_2(0,a+1)\in S_{10}$.
Otherwise, if $2\leq v\leq\frac{a}{3}-1$, then $1\leq v'\leq\frac{a}{3}-2$, and we obtain a lattice point $\alpha_2(v',r')$ in $S_{10}$ if $r'=a-v'+1$ and in $S_{11}$ if $a-v'+2\leq r'\leq a+1$. 
Finally, let us assume that $\frac{2a}{3}+1\leq r\leq a-2v$, which in particular implies that $1\leq v\leq\lfloor\frac{a-3}{6}\rfloor<\frac{a}{3}-1$. Let $r'=r-\frac{2a}{3}-1$. If $r=\frac{2a}{3}+1$, then $r'=0$ and the second summand above equals $\alpha_2(v,0)\in S_2$.
Otherwise, if $\frac{2a}{3}+2\leq r\leq a-2v$, then $1\leq r'\leq\frac{a}{3}-2v-1$, and we obtain $\alpha_2(v,r')\in S_1$.

We have seen that all points in $\blacktriangle_0(a)$ other than the ones in Theorem \ref{teolistindec} (ii) can be written as a sum of totally positive elements, so they are decomposable.
\end{proof}

In the cases with $s\neq0$, $T_1$ also preserves decompositions in both directions. Hence, we can use it to reduce the list of lattice points for which we have to find a decomposition.

\begin{pro}\label{pro:decs1} The elements with $s=1$ other than the ones in Theorem \ref{teolistindec} (v) are decomposable.
\end{pro}
\begin{proof}
Note that the points in Theorem \ref{teolistindec} (v) are just the ones in $R_4$. By Lemma \ref{lemat1s=1}, it is enough to prove the decomposability of points in $R_1$, $R_2$, $R_3$, $R_5$, $R_6$, $R_7$ and $R_{13}-\{v=\frac{2a}{3},\,\frac{a}{3}+2\leq r\leq a+1\}$. 

We first consider the points of the form $$\alpha=-(v-1)-(v(a+2)+r)\rho+(v+1)\rho^2-g_3,$$ which correspond to the regions $R_1$, $R_2$ and $R_3$. We have the decompositions \begin{equation*}
    \begin{split}
        \alpha&=[1+\rho^2-2g_3]+[-(v-1)-(v(a+2)+r-1)\rho+(v+1)\rho^2-2g_3]=\alpha_2(0,0)+P,\\\alpha&=\Big[-\frac{2a}{3}\rho+\rho^2-g_3\Big]+\Big[-(v-1)-\Big((v-1)(a+2)+r+\frac{a}{3}+2\Big)\rho+v\rho^2\Big]\\&=\alpha_1\Big(0,\frac{2a}{3}\Big)+Q.
    \end{split}
\end{equation*} We will prove that either $P$ or $Q$ are totally positive depending on the values of $v$ and $r$. Let us suppose that $\alpha\in R_1\cup R_2$, so $0\leq v\leq\frac{2a}{3}-1$ and $1\leq r\leq\mathrm{min}(\frac{a}{3},\frac{2a}{3}-v)$. If $v=0$ and we call $r'=r-1$, then $P=1-r'\rho+\rho^2-2g_3$ with $0\leq r'\leq\frac{a}{3}-1$. For $r=1$, we obtain $P=\alpha_2(0,0)\in S_2$, with $(e_1,e_2)=(1,1)$, 
while for $2\leq r\leq\frac{a}{3}$, we obtain $P=\alpha_2(0,r')\in S_1$.
Next, suppose that $1\leq v\leq\frac{2a}{3}-1$. Let us call $v'=v-1$ and $r'=r+\frac{a}{3}+2$. Then $Q=-v'-(v'(a+2)+r')\rho+(v'+1)\rho^2$, with $0\leq v'\leq\frac{2a}{3}-2$ and $\frac{a}{3}+3\leq r'\leq\mathrm{min}(\frac{2a}{3},a-v'+1)$.
Then $Q=\alpha_0(v',r')\in\blacktriangle$ is totally positive. Finally, we assume that $\alpha\in R_3$, so that $0\leq v\leq\frac{2a}{3}+1$ and $r=0$. For $v=0$, we have $\alpha=1+\rho^2-g_3=\alpha_1(0,0)$, which we already know is decomposable, and $T_1(\alpha_1(0,0))=\alpha_1(\frac{2a}{3}+1,0)$, so we may assume that $1\leq v\leq\frac{2a}{3}$. For $v'=v-1$ and $r'=\frac{a}{3}+2$, we have $Q=-v'-(v'(a+2)+r')\rho+(v'+1)\rho^2$, with $0\leq v'\leq\frac{2a}{3}-1$ and $0<r'<a-v'+1$. Then $Q=\alpha_0(v',r')\in\blacktriangle$ is totally positive.

Next, let us take $\alpha\in R_5\cup R_6\cup R_7$, so that $$\alpha=-(v-1)-(v(a+2)+r)\rho+(v+2)\rho^2-g_3$$ with $\frac{a}{3}+1\leq v\leq\frac{2a}{3}$ and $\frac{2a}{3}-v+1\leq r\leq\frac{a}{3}$ or $\frac{2a}{3}+1\leq v\leq a$ and $1\leq r\leq a-v+1$. Let us write $$\alpha=\alpha_2(0,0)+[-(v-1)-(v(a+2)+r-1)\rho+(v+2)\rho^2-2g_3].$$ On the other hand, our hypotheses on $v$ and $r$ imply that in particular $\frac{a}{3}\leq v\leq a$ and $1\leq r\leq a-v+1$, so the second summand is just $\alpha_2(v,r)\in S_4\cup S_5$, and hence it is totally positive. Therefore, $\alpha$ is decomposable for this case.

It remains to study the points in the region $R_{13}-\{v=\frac{2a}{3},\,\frac{a}{3}+2\leq r\leq a+1\}$. That is, we consider $$\alpha=-v-(v(a+2)+r)\rho+(v+2)\rho^2-g_3,$$ with $1\leq v\leq\frac{2a}{3}-1$ and $a-v+2\leq r\leq a+1$. Let us write
\begin{equation*}
    \begin{split}
        \alpha&=\Big[-\Big(\frac{a}{3}+1\Big)\rho+\rho^2-g_3\Big]+[-v-\Big(v(a+2)+r-\frac{a}{3}-1\Big)\rho+(v+1)\rho^2]
    \end{split}
\end{equation*}
The first summand is $\alpha_1(0,\frac{a}{3}+1)$, which is totally positive. As for the second one, let $r'=r-\frac{a}{3}-1$ and assume that $r\leq\frac{4a}{3}-v+1$ (note that this does not impose any restriction if $1\leq v\leq\frac{a}{3}$). Then $\frac{2a}{3}-v+1\leq r'\leq\mathrm{min}(\frac{2a}{3},a-v)$, so we obtain $\alpha_0(v,r')\in\blacktriangle$, which is totally positive. Hence, $\alpha$ is decomposable. The remaining case corresponds to the points for which $\frac{a}{3}+1\leq v\leq\frac{2a}{3}-1$ and $\frac{4a}{3}-v+2\leq r\leq a+1$. But those are the images by $T_1$ of the points in $R_{13}$ for which $1\leq v\leq\frac{a}{3}$, so they are decomposable.
\end{proof}

\begin{pro}\label{pro:decs2} The elements with $s=2$ other than the ones in Theorem \ref{teolistindec} (vi), (vii) and (viii) are decomposable.
\end{pro}
\begin{proof}
Among these lattice points, the indecomposables correspond to:
\begin{itemize}
    \item the line $\{v=0,\,0\leq r\leq\frac{a}{3}-1\}$, contained in $S_1\cup S_2$,
    \item the line $\{0\leq v\leq\frac{a}{3}-1,\,r=\frac{2a}{3}+1\}$, contained in $S_9$,
    \item the line $\{\frac{a}{3}\leq v\leq\frac{2a}{3}-1,\,r=\frac{4a}{3}-v+1\}$, contained in $S_{12}$.
\end{itemize} and we know that they are indecomposable from Proposition \ref{proindecs12}. Moreover, by Lemma \ref{lemat1s=2}, it is enough to find decompositions for all the other points in the regions $S_1$, $S_2$, $S_3$, $S_4$, $S_5$, $S_8$, $S_{10}$, $S_{11}$ and $S_{13}$.

We start with the regions $S_1$ and $S_2$. Among their points, the line $\{\alpha_2(0,r)\,|\,0\leq r\leq\frac{a}{3}-1\}$ are the elements in Theorem \ref{teolistindec} (vi), and we know that they are indecomposable from Proposition \ref{proindecs12}. The remaining points are of the form \begin{equation*}
    \begin{split}
        \alpha&=-(v-1)-(v(a+2)+r)\rho+(v+1)\rho^2-2g_3\\&=\alpha_1\Big(0,\frac{2a}{3}\Big)+\Big[-(v-1)-\Big(v(a+2)+r-\frac{2a}{3}\Big)\rho+v\rho^2-g_3\Big]
    \end{split}
\end{equation*} with $1\leq v\leq\frac{a}{3}-1$ and $0\leq r\leq\frac{a}{3}-v-1$.
We rewrite the second summand as $$-v'-\Big(v'(a+2)+r'\Big)\rho+(v'+1)\rho^2-g_3,$$ where $v'=v-1$ and $r'=r+\frac{a}{3}+2$.
Since $0\leq v'\leq\frac{a}{3}-2$ and $\frac{a}{3}+2\leq r'\leq\frac{2a}{3}-v'$, this is one of the indecomposables in Theorem \ref{teolistindec} (v).

Next, we consider the points from the regions $S_3$, $S_4$, $S_5$ and $S_8$, which are of the form $$\alpha=-(v-1)-(v(a+2)+r)\rho+(v+2)\rho^2-2g_3$$ with $0\leq v\leq\frac{a}{3}-1$ and $\frac{a}{3}-v\leq r\leq\frac{2a}{3}$ or $\frac{a}{3}\leq v\leq a+1$ and $0\leq r\leq a+1-v$. First we assume that $r>0$, so that $v\leq a$. Let us write $$\alpha=[1+\rho^2-2g_3]+[-v-(v(a+2)+r)\rho+(v+1)\rho^2].$$ Under these hypotheses we have $1\leq r\leq a-v+1$, so the second summand belongs to $\blacktriangle$ and hence it is totally positive. Now, let us choose $r=0$ and write $$\alpha=\Big[-\Big(\frac{a}{3}+1\Big)\rho+\rho^2-g_3\Big]+\Big[-(v-1)-\Big((v-1)(a+2)+\frac{2a}{3}+1\Big)\rho+(v+1)\rho^2-g_3\Big].$$ The first summand is $\alpha_1(0,\frac{a}{3}+1)$. The second one can be rewritten as $$-v'-\Big(v'(a+2)+r'\Big)\rho+(v'+2)\rho^2-g_3,$$ where $v'=v-1$ and $r'=\frac{2a}{3}+1$. This belongs to:
\begin{itemize}
    \item $R_{10}$, if $v=\frac{a}{3}$;
    \item $R_{12}$, if $v=\frac{a}{3}+1$;
    \item $R_{13}$, if $\frac{a}{3}+2\leq v\leq\frac{2a}{3}+1$;
    \item $R_{14}$, if $\frac{2a}{3}+2\leq v\leq a+1$.
\end{itemize}
Hence, the points in these regions are decomposable.

Now, we consider the points from the regions $S_{10}$ and $S_{11}$, which are of the form $$\alpha=-v-(v(a+2)+r)\rho+(v+2)\rho^2-2g_3$$ with $0\leq v\leq\frac{a}{3}-1$ and $a-v+1\leq r\leq a+1$. Let us write $$\alpha=\Big[-\Big(v+r-\frac{2a}{3}\Big)\rho+\rho^2-g_3\Big]+\Big[-v-\Big(v(a+2)+\frac{2a}{3}-v\Big)\rho+(v+1)\rho^2-g_3\Big].$$ The second summand is clearly the point $\alpha_1(v,\frac{2a}{3}-v)\in R_4$, and consequently totally positive. As for the first one, if we define $r'=v+r-\frac{2a}{3}$, from $a-v+1\leq r\leq a+1$ we obtain that $\frac{a}{3}+1\leq r'\leq\frac{a}{3}+1+v$, and since $v\leq\frac{a}{3}-1$ we have actually that $\frac{a}{3}+1\leq r'\leq\frac{2a}{3}$. Therefore the first summand is $\alpha_1(0,r')\in R_4$. Hence it is totally positive, and $\alpha$ is decomposable.

Finally, we consider the points of the region $S_{13}$, which are of the form $$\alpha=-(v-1)-(v(a+2)+r)\rho+(v+3)\rho^2-2g_3,$$ where $\frac{2a}{3}+1\leq v\leq a$ and $\frac{4a}{3}-v+2\leq r\leq\frac{2a}{3}+1$. We have that $$\alpha=[1+\rho^2-2g_3]+[-v-(v(a+2)+r)\rho+(v+2)\rho^2].$$ Since $r>a-v+2$, the second summand is a lattice point $\alpha_0(v,r)\in P_3$, so it is totally positive, and $\alpha$ is decomposable.
\end{proof}

We conclude that Theorem \ref{teolistindec} is established from Propositions  \ref{proindecs12}, \ref{promintrace2}, \ref{proindecs0}, \ref{pro:decs0}, \ref{pro:decs1} and \ref{pro:decs2}.

\section{Consequences}\label{sect:consequences}

In this section, we will use our knowledge of indecomposable integers for $[\O_K:\Z[\rho]]=3$ to derive several other results for these fields. First, we will focus on the minimal norm of algebraic integers in $K$ not associated with rational integers. Such a minimal norm for monogenic simplest cubic fields is $2a+3$ as was shown by Lemmermeyer and Peth\"o \cite{lemmermeyerpetho}. We will prove that except for a few cases of $a$, $2a+3$ is still minimal for our family with $p=3$.

Then, we will find an upper bound on the norm of indecomposable integers in $\O_K$. Recall that in any totally real number field, the norm of indecomposable integers is bounded \cite{brunotte,kalayatsynabound}. So far, this upper bound was intensively studied for real quadratic fields \cite{dressscharlau,jangkim,kalanorms,tinkovavoutier} and several families of monogenic totally real cubic number fields \cite{tinkova}. Let us highlight that our bound differs from the bound given in \cite{tinkova} for the simplest cubic fields with $\O_K=\Z[\rho]$.

Moreover, we will use indecomposable integers to show 
that the Pythagoras number of $\O_K$ is always $6$ in our subfamily. A similar result was derived by Tinkov\'a \cite{tinkovapyth}, who proved that the Pythagoras number of $\Z[\rho]$ is $6$ whenever $a\geq 3$. Moreover, in the final subsection, we will find both upper and lower bounds on the minimal number of variables of universal quadratic forms over $\O_K$. For that, we will follow the procedure developed in \cite{kalatinkova}, which is based on the knowledge of indecomposable integers and their minimal traces after multiplication by elements of the codifferent.

\subsection{The smallest norm}

Now, we will find the smallest norm of elements which are not associated with rational integers. To do that, we use the knowledge of the structure of indecomposables stated in Theorem \ref{teolistindec}.

\begin{pro}\label{prosmallestnorm}
Let $\alpha\in\O_K$. Then either 
\[
|N(\alpha)|\geq\left\{
\begin{array}{ll} \frac{\Delta}{27} & \text{ if } a=21,30,48, \\
 2a+3 & \text{ if } a>48,
\end{array}\right.
\]
or $\alpha$ is associated with a rational integer. Moreover, this lower bound is attained by some $\alpha\in\O_K$.
\end{pro}

\begin{proof}
The smallest such norm can be attained either by an indecomposable integer or by a sum of two totally positive units.
Let us first discuss the second case. Recall that even in this case, a system of fundamental units is formed by the pair $\rho$ and $\rho'$, see Corollary \ref{coro:fundamentalunits}. Every sum of two units is associated with an element of the form $1+\varepsilon$ where $\varepsilon$ is a totally positive unit. Moreover, we have $\varepsilon\neq 1$ since otherwise, $1+\varepsilon$ would be associated with the rational integer $2$. Therefore, by \cite[Lemma 6.2]{kalatinkova}, $\varepsilon$ is greater than $a^2$ in some embedding. Thus,
\[
N(1+\varepsilon)=N(1)+N(\varepsilon)+\text{Tr}(\varepsilon)+\text{Tr}(\varepsilon\varepsilon')>a^2.
\] 
However, $a^2>2a+3$ for $a\geq 21$, and $2a+3$ is the norm of one of the indecomposable integers in $\O_K$. Thus, a sum of two units cannot have the smallest norm.

Since the elements in Theorem \ref{teolistindec} (iii) and (iv) are associated with conjugates of elements in \ref{teolistindec} (ii), and the same is true for (vi)--(viii), it is enough to consider elements in (i), (ii), (v) and (vi).

The norm of $\frac{1+\rho+\rho^2}{3}$ is equal to $\frac{a^2+3a+9}{27}$. The smallest norm among elements in (ii) is $2a+3$.

Regarding elements in (vi), we have
\[
N\left(-(r+1)g_2+g_3\right)=-r^3-3r^2+\frac{a^2+3a-18}{9}r+\frac{4a^2+12a+9}{27}
\]
where $0\leq r \leq \frac{a}{3}-1$.
This norm is a cubic polynomial in $r$ with a negative leading coefficient. Moreover, for $r=-2$, it is equal to $-\frac{2a^2+6a-9}{27}<0$ for $a\geq 21$. It means that in the interval $\left[0,\frac{a}{3}-1\right]$, it attains its smallest value at $0$ or at $\frac{a}{3}-1$. For $r=0$, we obtain $\frac{4a^2+12a+9}{27}>2a+3$ for $a\geq 21$. Similarly, for $r=\frac{a}{3}-1$, we again obtain the norm $\frac{4a^2+12a+9}{27}$. Thus, the smallest norm is not attained by elements in (vi). 

Let us now focus on elements in (v). We can deduce that
\begin{multline*}
N\left(-(2v+1)g_1-(v(a+3)+r+1)g_2+(3v+2)g_3\right)\\=-r^3-(av+3v-a)r^2-\left(av^2-a^2v-3av-3v+\frac{2a^2}{9}-\frac{a}{3}-1\right)r\\+v^3+\frac{a^2}{3}v^2-\frac{2a^3+7a^2+12a+9}{9}v-\frac{4a^2+12a+9}{27}
\end{multline*}
where $0\leq v\leq \frac{a}{3}-1$ and $\frac{a}{3}+1\leq r\leq \frac{2a}{3}-v$.
This is again a cubic polynomial in $r$ with a negative leading coefficient. If $r=\frac{a}{3}$, we obtain 
\[
v^3-\frac{a^3+3a+9}{9}v-\frac{a^2+3a+9}{27},
\]
which is a cubic polynomial in $v$ with a positive leading coefficient. Moreover, for $v=-1$, it is $\frac{2a^2+6a-9}{27}>0$; for $v=0$, it is equal to $-\frac{a^2+3a+9}{27}<0$, and for $v=\frac{a}{3}-1$, it equals $-\frac{10a^2-24a+9}{27}<0$ for $a\geq 21$. It gives that whenever $r=\frac{a}{3}$ and $0\leq v\leq \frac{a}{3}-1$, then our norm is negative. It implies that for a fixed $v$, the smallest norm is attained either for $r=\frac{a}{3}+1$, or for $r=\frac{2a}{3}-v$.

Let us now discuss the case when $r=\frac{a}{3}+1$. Here, we get the norm    
\[
v^3-av^2+\frac{2a^2-3a-9}{9}v+\frac{2a^2+6a-9}{27}.
\]
For $v=\frac{a}{3}$, it is equal to $-\frac{a^2+3a+9}{27}<0$. Thus, it again attains its smallest value in one of the border points of our interval. For both of them, it is equal to $\frac{2a^2+6a-9}{27}>2a+3$ for $a\geq 26$. If $a=21$, then $\frac{2a^2+6a-9}{27}=37>19=\frac{a^2+3a+9}{27}$, so we can exclude this norm even in this case.   

Similarly, for $r=\frac{2a}{3}-v$, we obtain
\[
-v^3-3v^2+\frac{a^2+3a-18}{9}v+\frac{2a^2+6a-9}{27}.
\] 
For $v=-1$, it is equal to $-\frac{a^2+3a+9}{27}$; thus, the smallest norm is attained in a border point. For both of them, we again obtain the value $\frac{2a^2+6a-9}{27}$ which we have excluded before.

Thus, we are left with the norms $\frac{a^2+3a+9}{27}$ and $2a+3$, and $\frac{a^2+3a+9}{27}>2a+3$ if $a\geq 53$. Between $21$ and $52$, only the cases $a=21,30,48$ belong to our family, and for them, the norm     $\frac{a^2+3a+9}{27}$ is the smallest one.
\end{proof}

\subsection{The largest norm}

In this part, we will find a sharp upper bound on the norm of indecomposable integers in $\O_K$. As in the case of the smallest norm, it is enough to discuss elements in (i), (ii), (v) and (vi). We will start with elements in (ii) and (vi).

\begin{lema} \label{prop:normlines}
Let $\alpha\in\O_K$ be an indecomposable integer. Then:
\begin{enumerate}
\item If $\alpha$ is as in (ii), then $N(\alpha)\leq \frac{2a^3+9a^2+27a+27}{27}=\frac{(2a+3)\Delta}{27}$.
\item If $\alpha$ is as in (vi), then $N(\alpha)< \frac{(2a+3)\Delta}{27}$. 
\end{enumerate}
\end{lema}

\begin{proof}
To prove (1), it is not difficult to see that the largest norm among these elements is attained when $r=\frac{a}{3}$, for which $N(\alpha)=\frac{2a^3+9a^2+27a+27}{27}$.

On the other hand, if $\alpha=-(r+1)g_2+g_3$, we have
\[
g(r)=\frac{2a^3+9a^2+27a+27}{27}-N(\alpha)=r^3+3r^2-\frac{a^2+3a-18}{9}r+\frac{2a^3+5a^2+15a+18}{27}.
\]
This is a cubic polynomial in $r$ with a positive leading coefficient. Moreover, the discriminant of $g$ is negative, in which case $g$ has only one real root. This together with $g(0)=\frac{2a^3+5a^2+15a+18}{27}>0$ implies that $g(r)>0$ for all $0\leq r\leq \frac{a}{3}-1$, completing the proof. 
\end{proof}

Note that 
\[
N\left(\frac{1+\rho+\rho^2}{3}\right)=\frac{a^2+3a+9}{27}<\frac{2a^3+9a^2+27a+27}{27}
\]
for $a\geq 21$. Therefore, we can also exclude the element $\frac{1+\rho+\rho^2}{3}$ from the consideration.

It remains to discuss the elements in (v), which form the triangle $R_4$ in Section \ref{sec:reprpoints}. Let us denote these integers as
\[
\alpha(v,r)=-(2v+1)g_1-(v(a+3)+r+1)g_2+(3v+2)g_3
\]
where $0\leq v\leq \frac{a}{3}-1$ and $\frac{a}{3}+1\leq r\leq \frac{2a}{3}-v$.
Furthermore, for the transformation $T_1$ in Section \ref{sec:decspoints} we have $T_1(R_4)=R_4$. In other words, $R_4$ also contains $\alpha'\varepsilon_1$ and $\alpha''\varepsilon_2$ where $\varepsilon_1$ and $\varepsilon_2$ are concrete totally positive units. Thus, it suffices to consider only one third of this triangle, in particular,
\[
R_4^0=\left\{\alpha(v,r);0\leq v\leq \frac{a-3}{9}-1,\frac{a}{3}+1+v\leq r\leq \frac{2a}{3}-2v-1\right\}\cup\left\{\alpha\left(\frac{a-3}{9},\frac{a}{3}+1+\frac{a-3}{9}\right)\right\}.
\]

The following lemma compares norms of elements contained in $R_4^0$. Note that the proof is almost the same as in \cite[Lemma 3.2]{tinkova}.

\begin{lema} \label{lem:comnorms}
	We have 
	\begin{enumerate}
		\item $N(\alpha(v,r))<N(\alpha(v+1,r))$ for all $0\leq v \leq \frac{a-3}{9}-1$ and $\frac{a}{3}+2+v\leq r\leq \frac{2a}{3}-2v-3$,
		\item $N\big(\alpha\big(v,\frac{a}{3}+1+v\big)\big)<N\big(\alpha\big(v+1,\frac{a}{3}+1+v+1\big)\big)$ for all $0\leq v \leq \frac{a-3}{9}-1$,
		\item $N\big(\alpha\big(v,\frac{2a}{3}-2v-1\big)\big)<N\big(\alpha\big(v+1,\frac{2a}{3}-2(v+1)-1\big)\big)$ for all $0\leq v \leq \frac{a-3}{9}-2$,
		\item $N\big(\alpha\big(v,\frac{2a}{3}-2v-2\big)\big)<N\big(\alpha\big(v+1,\frac{2a}{3}-2(v+1)-2\big)\big)$ for all $0\leq v \leq \frac{a-3}{9}-2$.   
	\end{enumerate}
\end{lema}

Now we have everything we need to determine the upper bound on the norm of indecomposable integers in $\O_K$.

\begin{pro}\label{pro:largestnorm}
If $\alpha$ is indecomposable in $\O_K$, then 
\[
|N(\alpha)|\leq\left\{
\begin{array}{ll} \frac{2a^3+9a^2+27a+27}{27} \text{ if } a=21, 30, 48, \\
 \frac{(a^2+3a+9)^2}{729} \text{ if } a>48,
\end{array}\right.
\]
\end{pro} 

\begin{proof}
As we have seen, the largest norm is either $\frac{2a^3+9a^2+27a+27}{27}$ found in Proposition \ref{prop:normlines} or is attained by some elements from $R_4^0$. Moreover, applying Lemma \ref{lem:comnorms}, we are left only with two candidates from $R_4^0$: $\alpha\left(\frac{a-3}{9}-1,\frac{a}{3}+2+\frac{a-3}{9}\right)$ and $\alpha\left(\frac{a-3}{9},\frac{a}{3}+1+\frac{a-3}{9}\right)$. Furthermore,
\begin{multline*}
N\left(\alpha\left(\frac{a-3}{9}-1,\frac{a}{3}+2+\frac{a-3}{9}\right)\right)=\frac{a^4+6a^3-54a^2-189a+81}{729}\\<\frac{(a^2+3a+9)^2}{729}=N\left(\alpha\left(\frac{a-3}{9},\frac{a}{3}+1+\frac{a-3}{9}\right)\right).
\end{multline*}
Moreover,
\[
\frac{2a^3+9a^2+27a+27}{27}<\frac{(a^2+3a+9)^2}{729}
\]  
for $a\geq 53$, which gives a few exceptional cases listed in the statement of the proposition.
\end{proof}

\subsection{Pythagoras number}

In this part, we will show that the Pythagoras number of $\O_K$ is $6$. Before that, let us briefly overview basic facts and results on this number. Consider a commutative ring $\O$, for which we define the following two sets: 
\[
\sum\O_K^2=\Bigg\{\sum_{i=1}^n \alpha_i^2; \alpha_i\in\O_K \text{ and }n\in\N\Bigg\}
\]
and for each $m\in\mathbb{N}$
\[
\sum^m\O_K^2=\Bigg\{\sum_{i=1}^m \alpha_i^2; \alpha_i\in\O_K\Bigg\}.
\]
Then by the Pythagoras number $\P(\O)$ of $\O$, we will mean the following infimum:
\[
\P(\O)=\inf\Bigg\{m\in\N\cup\{\infty\};\sum\O_K^2=\sum^m\O_K^2\Bigg\}.
\]
For example, $\P(\C)=\P(\R)=1$ and $\P(\Z)=\P(\Q)=4$. Although there are many results on the Pythagoras number of fields, we will only summarize results in the case when $\O$ is an order of a totally real number field $K$. 

For them, we know that the value of $\P(\O)$ is finite but can attain arbitrarily large values \cite{scharlau}. Moreover, in this case, the Pythagoras number can be bounded by a function depending on the degree $d$ of a field $K$, and, furthermore, $\P(\O)\leq d+3$ for $2\leq d\leq 5$ \cite{kalayatsyna}. So far, the Pythagoras number was fully determined for orders in real quadratic fields \cite{peters}, and we have some partial results for real biquadratic fields \cite{krasenskyraskasgallova,hehu} and totally real cubic fields \cite{krasensky,tinkovapyth, kst}. In particular, we closely follow \cite{tinkovapyth}, where the second author showed that $\P(\Z[\rho])=6$ for $\rho$ being a root of the polynomial $x^3-ax^2-(a+3)x-1$ with $a\geq 3$, i.e., for a concrete order of the simplest cubic fields.

In cubic fields, we know that $\P(\O_K)\leq 6$ by the result of Kala and Yatsyna \cite{kalayatsyna}, and we will prove that $\P(\O_K)\geq 6$ for our subfamily of non-monogenic simplest cubic fields with the base $B_3(1,1)$.
For that, we will use the method developed in \cite{tinkovapyth}, and our determination of the structure of indecomposable integers in $\O_K$. In particular, we will do the following:
\begin{enumerate}
\item We will choose a suitable element $\gamma\in\O_K$ such that $\gamma$ can be written as a sum of six non-zero squares of elements in $\O_K$. 
\item We will determine all elements $\alpha$ from Theorem \ref{teolistindec} and all units $\varepsilon\in\O_K$ such that $\gamma\succeq (\varepsilon\alpha)^2$.\label{meth:inde}
\item Using elements in (\ref{meth:inde}), we will find all elements $\omega\in\O_K$ such that $\gamma\succeq \omega^2$. \label{meth:all}
\item Considering elements in (\ref{meth:all}), we will discuss all possible decompositions of $\gamma$ as a sum of squares and show that for every such decomposition, we need at least $6$ squares.
\end{enumerate}
Note that in fact, elements $\varepsilon\alpha$ in (\ref{meth:inde}) are indecomposable integers in other signatures. Since the simplest cubic fields have units of all signatures, they can be expressed as presented here, i.e., as $\varepsilon\alpha$ where $\varepsilon$ is a (not necessarily totally positive) unit and $\alpha$ is one of totally positive indecomposables listed in Theorem \ref{teolistindec}. For more details, see, for example, \cite{tinkovapyth}.  

In our proof, we will fix $\gamma$ as 
\begin{align*}
\gamma&=1+1+1+4+\left(\frac{a+6}{3}g_1+\frac{a}{3}g_2-g_3\right)^2+\left(\frac{5a+3}{9}g_1+\frac{2a+3}{3}g_2-2g_3\right)^2\\
&=\frac{34a^2+15a+783}{81}+\frac{11a^2-29a-39}{27}\rho-\frac{11a-33}{27}\rho^2.
\end{align*}
Using estimates on $\rho,\rho'$, and $\rho''$ (see Section \ref{sect:prelim}; note that for $\rho'$, we use bounds $-1-\frac{1}{a+1}<\rho'<-1-\frac{1}{a+2}$ found in \cite{tinkovapyth}), we obtain
\begin{align*}
\gamma&<\frac{13a^2}{81}+\frac{14a}{27}+\frac{197}{27}-\frac{26}{9a}<a^2,\\
\gamma'&<\frac{a^4+40a^3+1234a^2+4596a+4725}{81(a+2)^2}<a^2,\\
\gamma''&<\frac{34a^4+186a^3+1167a^2+5178a+7497}{81(a+3)^2}<a^2
\end{align*}
for $a\geq 21$.
We see that if $\omega\in\Z$ is such that $\gamma\succeq \omega^2$, then $|\omega|<a$. In contrast with \cite{tinkovapyth}, we will not determine the precise set of elements $\omega$ satisfying $\gamma\succeq \omega^2$, but we will significantly restrict the set of such elements. As we will see, this restriction will be enough to show that we need at least $6$ squares to express $\gamma$.

First of all, in a series of lemmas, we will find all $\omega\in\O_K$ such that $\gamma\succeq \omega^2$. We will start with units and the indecomposable integer $g_3=\frac{1+\rho+\rho^2}{3}$. Recall that in the simplest cubic fields, all totally positive units are squares. Moreover, conjugates of $g_3$ are associated with $g_3$. 

\begin{lema} \label{lem:pythunitexcept}
We have
\begin{enumerate}
\item if $\gamma\succeq\varepsilon^2$ where $\varepsilon$ is a unit in $\O_K$, then $\varepsilon^2=1$,
\item if $\gamma\succeq \left(\varepsilon\frac{1+\rho+\rho^2}{3}\right)^2$ where $\varepsilon$ is a unit in $\O_K$, then $\varepsilon^2=\rho'^2\rho''^2$.
\end{enumerate}
\end{lema}

\begin{proof}
Recall from Corollary \ref{coro:fundamentalunits} that a system of fundamental units is formed by the pair $\rho$ and $\rho'$. Thus, we can use \cite[Lemma 6.2]{kalatinkova}, which says that if $\varepsilon^2\neq 1$ is a unit in $\O_K$, then $\varepsilon^2$ has a conjugate greater than $a^2$. Since $\gamma$ has all conjugates smaller than $a^2$, no totally positive unit except for $1$ can be totally smaller than $\gamma$.  

Considering the second part of the statement, we see that
\begin{align*}
\left(\frac{1+\rho+\rho^2}{3}\right)^2&>\frac{(a^2+3a+3)^2}{9}>a^2,\\
\left(\frac{1+\rho'+\rho'^2}{3}\right)^2&>\frac{1}{9}\left(\frac{(a+3)^2}{(a+2)^2}-\frac{1}{a+1}\right)^2,\\
\left(\frac{1+\rho''+\rho''^2}{3}\right)^2&>\frac{1}{9}\left(1-\frac{1}{a+2}+\frac{1}{(a+3)^2}\right)^2.
\end{align*}
Thus, we can immediately exclude that $\varepsilon^2=1$.
Moreover, it can be easily shown that
\[
a^2\left(\frac{1+\rho'+\rho'^2}{3}\right)^2>\gamma' \qquad \text{ and }
\qquad a^4\left(\frac{1+\rho''+\rho''^2}{3}\right)^2>\gamma''.
\]
Here, we use \cite[Lemma 6.3]{kalatinkova}. It says that if $\varepsilon>a^2$ is a totally positive unit such that $\varepsilon\neq \rho^2,\rho^2\rho'^2$, then either $\varepsilon>a^4$, or one of $\varepsilon'$ and $\varepsilon''$ is greater than $a^2$. 
Together with previous estimates, it implies $\varepsilon''^2=\rho^2,\rho^2\rho'^2$ in our case, from which we can exclude $\varepsilon''^2=\rho^2$, which does not have $\varepsilon^2<1$.
\end{proof} 

We will proceed with indecomposables from $\Z[\rho]$.

\begin{lema} \label{lem:pythmono}
Let $\alpha$ be as in (ii), (iii) or (iv) of Theorem \ref{teolistindec}. Then there is no unit $\varepsilon\in\O_K$ such that $\gamma\succeq (\varepsilon\alpha)^2$. 
\end{lema}

\begin{proof}
Let $\alpha=-r\rho+\rho^2$ where $1\leq r\leq \frac{a}{3}$, i.e., $\alpha$ is one of elements in (ii).
In this case, we have 
\begin{align*}
(-r\rho+\rho^2)^2&>((a+1)^2-r(a+2))^2\geq \frac{(2a^2+4a+3)^2}{9}>a^2,\\
(-r\rho'+\rho'^2)^2&>(r+1)^2,\\
(-r\rho''+\rho''^2)^2&>\frac{r^2}{(a+3)^2}.
\end{align*}
Therefore, we must have $\varepsilon^2<1$, $\varepsilon'^2<a^2$ and $\varepsilon''^2>a^2$, which again follows from \cite[Lemma~6.2]{kalatinkova}.
Moreover, easily, $\frac{r^2a^4}{(a+3)^2}>\gamma,\gamma',\gamma''$. Therefore, again by \cite[Lemma 6.3]{kalatinkova}, we can conclude that $\varepsilon''^2=\rho^2,\rho^2\rho'^2$. However, for both of these units, we have $(\varepsilon\alpha)^2>\gamma,\gamma',\gamma''$.

Note that by this, we have also excluded Cases (iii) and (iv). It follows from the fact that elements in (iii) and (iv) are associated with conjugates in (ii), and, thus, $(\varepsilon\alpha)^2>\gamma,\gamma',\gamma''$ also covers elements originating from indecomposables in (iii) and (iv).   
\end{proof}

Now we will discuss the other three lines of indecomposables. 

\begin{lema} \label{lem:pythline}
Let $a\geq 55$.
Let $\alpha$ be as in (vi), (vii) or (viii) of Theorem \ref{teolistindec}, and let $\gamma\succeq (\varepsilon\alpha)^2$ for some unit $\varepsilon$. Then $\alpha$ is as in (vi) and $\varepsilon^2=\rho'^2\rho''^2$.  
\end{lema}

\begin{proof}
Let $\alpha=-(r+1)g_2+g_3$ where $0\leq r\leq \frac{a}{3}-1$, i.e., $\alpha$ is one of elements in (vi). Then we have
\begin{align*}
\alpha^2&>\left(\frac{a^2}{3}-\frac{4}{3a}-r\Big(a+1+\frac{2}{a}\Big)\right)^2>\left(\frac{2a}{3}+\frac{1}{3}+\frac{2}{3a}\right)^2>\gamma,\gamma',\gamma'',\\
\alpha'^2&>\left(\frac{4}{3}+\frac{4}{3(a+2)}+\frac{1}{3(a+2)^2}+r\Big(1+\frac{1}{a+2}\Big)\right)^2>1,\\
\alpha''^2&>\left(\frac{1}{3}+\frac{2}{3(a+3)}+\frac{1}{3(a+3)^2}+\frac{r}{a+3}\right)^2.
\end{align*}
Moreover, $a^4\alpha''^2>\gamma,\gamma',\gamma''$. 
It implies that again $\varepsilon''^2=\rho^2,\rho^2\rho'^2$ by \cite[Lemma 6.3]{kalatinkova}. We can exclude $\varepsilon''^2=\rho^2$ as $\varepsilon^2>1$. 

Now we will discuss all three conjugates of $(\rho'\rho''\alpha)^2$.
Since $(\rho\rho'\alpha'')^2>\gamma'$, the conjugate $(\rho''\rho\alpha')^2$ cannot be totally smaller than $\gamma$. If $r\geq \frac{a}{6}$, then $(\rho\rho'\alpha'')^2>\gamma$, and if $r<\frac{a}{6}$, then $(\rho'\rho''\alpha)^2>\gamma'$ for $a\geq 55$. This excludes the conjugate $(\rho\rho'\alpha'')^2$. Thus, we are left only with $(\rho'\rho''\alpha)^2$, which states the lemma.

Note that similarly as in the proof of Lemma \ref{lem:pythmono}, conjugates of elements in (vi) are associated with elements in (vii) and (viii), and in this proof, we have discussed all conjugates of $\alpha$ at the same time.     
\end{proof}

In the end, we will look at the triangle of indecomposables in (v). For that, let us denote
\[
\alpha(v,r)=-(2v+1)g_1-(v(a+3)+r+1)g_2+(3v+2)g_3
\]
where $0\leq v\leq \frac{a}{3}-1$ and $\frac{a}{3}+1\leq r\leq \frac{2a}{3}-v$. 

\begin{lema} \label{lem:pythtriangle}
Let $a\geq 63$. Moreover, let $\gamma\succeq (\varepsilon\alpha(v,r))^2$ for some unit $\varepsilon\in\O_K$, $0\leq v\leq \frac{a}{3}-1$ and $\frac{a}{3}+1\leq r\leq \frac{2a}{3}-v$. Then either
\begin{enumerate}
\item $(\varepsilon\alpha(v,r))^2=\rho''^2\rho^2\alpha(0,r)'^2$ for some $\frac{a}{3}+1\leq r< \frac{5a-6}{12}$, or
\item $(\varepsilon\alpha(v,r))^2=\rho^2\rho'^2\alpha(0,r)''^2$ for some $\frac{5(a-3)}{9}\leq r\leq \frac{2a}{3}-1$.  
\end{enumerate}
\end{lema}

\begin{proof}
As before, it suffices to restrict to a subset of the triangle of indecomposables in (v). In particular, we can discuss only elements in $R_4^0$.
Thus, let $\alpha(v,r)\in R_4^0$, i.e., $0\leq v\leq \frac{a-3}{9}-1$ and $\frac{a}{3}+1+v\leq r\leq \frac{2a}{3}-2v-1$, or $v=\frac{a-3}{9}$ and $r=\frac{a}{3}+1+\frac{a-3}{9}$. Then
\begin{align*}
\alpha(v,r)^2&>\left(\frac{2a^2}{3}+a-\frac{2}{3a}-v\Big(a+4+\frac{4}{a}\Big)-r\Big(a+1+\frac{2}{a}\Big)\right)^2\geq\left(\frac{4a}{3}-\frac{1}{3}+\frac{4}{3a}\right)^2>a^2,\\
\alpha(v,r)'^2&>\left(\frac{2}{3}+\frac{5}{3(a+2)}+\frac{2}{3(a+2)^2}+v\Bigg(a+2+\frac{a+4}{a+2}+\frac{1}{(a+2)^2}\Bigg)+r\Bigg(1+\frac{1}{a+2}\Bigg)\right)^2\\
&\hspace{5cm}\geq \left(\frac{a+5}{3}+\frac{a+8}{3(a+2)}+\frac{2}{3(a+2)^2}\right)^2>1,\\
\alpha(v,r)''^2&>\left(-\frac{1}{3}+\frac{1}{3(a+3)}+\frac{2}{3(a+3)^2}+v\Bigg(-1+\frac{a+2}{a+3}+\frac{1}{(a+3)^2}\Bigg)+\frac{r}{a+3}\right)^2
\geq \frac{(a+5)^2}{9(a+3)^4}.
\end{align*}
Furthermore, $a^6\alpha(v,r)''^2>\gamma,\gamma',\gamma''$. That means that our unit has to satisfy $\varepsilon^2<1$, $\varepsilon'^2<a^2$ and $a^2<\varepsilon''^2<a^6$. This is fulfilled by units $\varepsilon''^2=\rho^2\rho'^2,\rho^2\rho'^4,\rho^4\rho'^2,\rho^4\rho'^4,\rho^4\rho'^6$. 

Now, we will use the fact that if $\gamma\succeq\omega^2$ for some $\omega\in\O_K$, then necessarily $\text{Tr}(\gamma)\geq \text{Tr}(\omega^2)$.
It can be computed that $\text{Tr}(\gamma)=\frac{16a^2-24a+981}{27}$. Moreover, we have
\begin{align*}
\text{Tr}(\rho'^2\rho''^4\alpha(v,r)^2)&\geq \frac{a^4+12a^3+56a^2+132a+153}{9}>\text{Tr}(\gamma),\\
\text{Tr}(\rho'^4\rho''^6\alpha(v,r)^2)&\geq \frac{a^4+10a^3+41a^2+96a+126}{9}>\text{Tr}(\gamma),
\end{align*} 
so we can exclude $\varepsilon^2=\rho'^2\rho''^4,\rho'^4\rho''^6$. Moreover, if $v\geq 1$, then
\begin{align*}
\text{Tr}(\rho'^2\rho''^2\alpha(v,r)^2)&\geq \frac{17a^2+96a+369}{9}>\text{Tr}(\gamma),\\
\text{Tr}(\rho'^4\rho''^2\alpha(v,r)^2)&\geq \frac{26a^2+6a+153}{9}>\text{Tr}(\gamma),\\
\text{Tr}(\rho'^4\rho''^4\alpha(v,r)^2)&\geq \frac{41a^2+150a+234}{9}>\text{Tr}(\gamma),
\end{align*}
by which we have excluded all cases with $v\geq 1$.

Let us now focus on the case when $v=0$. If $r\geq \frac{a}{3}+2$, then
\begin{align*}
\text{Tr}(\rho'^4\rho''^2\alpha(0,r)^2)&\geq \frac{26a^2+42a+126}{9}>\text{Tr}(\gamma),\\
\text{Tr}(\rho'^4\rho''^4\alpha(0,r)^2)&\geq \frac{26a^2+114a+234}{9}>\text{Tr}(\gamma).
\end{align*}
If $r=\frac{a}{3}+1$, we obtain
\begin{equation} \label{eq:secondcoef}
(\gamma-\beta)(\gamma'-\beta')+(\gamma-\beta)(\gamma''-\beta'')+(\gamma'-\beta')(\gamma''-\beta'')<0
\end{equation}
where $\beta$ is any conjugate of $\rho'^4\rho''^2\alpha\big(0,\frac{a}{3}+1\big)^2$ or $\rho'^4\rho''^4\alpha\big(0,\frac{a}{3}+1\big)^2$. Note that (\ref{eq:secondcoef}) gives the second coefficient in the minimal polynomial of $\gamma-\beta$, which has to be nonnegative if $\gamma-\beta\succeq 0$. By this, we have excluded the units $\varepsilon^2=\rho'^4\rho''^2,\rho'^4\rho''^4$.

Thus, it remains to discuss the unit $\varepsilon^2=\rho'^2\rho''^2$. We have $\rho''^2\rho^2\alpha(0,r)'^2>\gamma'$ for all $\frac{a}{3}+1\leq r\leq \frac{2a}{3}-1$, which excludes the conjugate $\rho'^2\rho''^2\alpha(0,r)^2$. 

Let us now focus on the conjugate $\rho''^2\rho^2\alpha(0,r)'^2$. If $r\geq \frac{5a-6}{12}$, we have
\[
\rho''^2\rho^2\alpha(0,r)'^2>\frac{1}{(a+3)^2}(a+1)^2\frac{(5a^3+27a^2+52a+44)^2}{144(a+2)^4}>\gamma
\] 
for $a\geq 63$. Therefore, we must have $\frac{a}{3}+1\leq r< \frac{5a-6}{12}$, which gives the first case in our statement.

For the conjugate $\rho^2\rho'^2\alpha(0,r)''^2$, we similarly see that for $\frac{a}{3}+1\leq r\leq \frac{5(a-3)}{9}-1$, we obtain
\[
\rho'^2\rho''^2\alpha(0,r)^2>\left(1+\frac{1}{a+2}\right)^2\frac{1}{(a+3)^2}\frac{(a^3+28a^2+14a+42)^2}{81a^2}>\gamma'
\]
for $a\geq 34$. This implies $\frac{5(a-3)}{9}\leq r\leq \frac{2a}{3}-1$, i.e., the second case in our statement.  
\end{proof}

In Lemmas \ref{lem:pythunitexcept}, \ref{lem:pythmono}, \ref{lem:pythline} and \ref{lem:pythtriangle}, we have found some of the squares totally smaller than $\gamma$. There can exist other squares $\omega^2\in\O_K$ satisfying $\gamma\succeq \omega^2$; however, these squares can be determined from elements derived so far. To see this, we recall the following notions: 
\begin{itemize}
    \item By the signature of an element $\alpha\in K$, we mean the triple $(\text{sgn}(\alpha),\text{sgn}(\alpha'),\text{sgn}(\alpha''))$, where $\text{sgn}$ is the signum function. We will replace $1$ and $-1$ in this triple by symbols $+$ and $-$.
    \item For a signature $\sigma$, an element in $\mathcal{O}_K$ is $\sigma$-indecomposable if it cannot be written as a sum of two elements in $\mathcal{O}_K$ with signature $\sigma$ (see, for instance, \cite{tinkovapyth}).
\end{itemize} 
Every $\sigma$-decomposable integer $\omega\in\mathcal{O}_K$ is a sum of $\sigma$-indecomposable elements, say $\omega=\sum_{i=1}^n \beta_i$, each of which is associated with a totally positive indecomposable in Theorem \ref{teolistindec} (since simplest cubic fields have units of all signatures). Now, if we assume that $\gamma\succeq\omega^2$, then $\gamma\succeq\beta_i^2$, so the elements $\beta_i$ in this decomposition are square roots of squares found in the above lemmas.

Therefore,
\[
\gamma\succeq \omega^2=\left(\sum_{i=1}^n\beta_i\right)^2\succeq (\beta_i+\beta_j)^2\succeq \beta_i^2
\]
for every $1\leq i,j\leq n$. Thus, to determine all the remaining squares totally smaller than $\gamma$, it is enough for each possible signature 
\begin{enumerate}[(1)]
    \item to take all square roots of squares from Lemmas \ref{lem:pythunitexcept}, \ref{lem:pythmono}, \ref{lem:pythline} and \ref{lem:pythtriangle} with this signature, \label{signatureset}
    \item to study the element $(\beta_i+\beta_j)^2$ for all pairs of elements $\beta_i$ and $\beta_j$ from \ref{signatureset}.
    \begin{itemize}
    \item If $(\beta_i+\beta_j)^2\succ \gamma$ for all such pairs, then we do not get any additional square totally smaller than $\gamma$. In the following proof, we will sometimes show that $\beta_i^2,\beta_j^2\succ \frac{1}{2}\gamma$, from which $(\beta_i+\beta_j)^2\succ \gamma$ follows. \item Otherwise, we can get more squares. However, we will see that for $\gamma$, this can occur only in one trivial case.    
    \end{itemize}
\end{enumerate}
Moreover, $\alpha$ and $-\alpha$ produce the same square, so it suffices to study only one of every pair of (mutually opposite) signatures; e.g., instead of studying both elements with signatures $(+,+,+)$ and $(-,-,-)$, it is enough to examine those with the signature $(+,+,+)$.

\begin{lema} \label{lem:pythdecompo}
Let $a\geq 63$ and let $\gamma\succeq \omega^2$ where $\omega^2\in \O_K$ is not as in Lemmas \ref{lem:pythunitexcept}, \ref{lem:pythmono}, \ref{lem:pythline} and \ref{lem:pythtriangle}. Then $\omega\in\Z$.
\end{lema}

\begin{proof} Recall that the unit $\rho$ has the signature $(+,-,-)$. Thus, the so far found elements, whose squares are totally smaller than $\gamma$, have the following signatures: 
\begin{enumerate}
\item $1$ has the signature $(+,+,+)$, \label{el:1}
\item $\rho'\rho''\frac{1+\rho+\rho^2}{3}$ has the signature $(+,-,-)$, \label{el:2}
\item $\rho'\rho''\left(-(r+1)g_2+g_3\right)$ where $0\leq r\leq \frac{a}{3}-1$ have the signature $(+,-,-)$, \label{el:3}
\item $\rho''\rho\alpha(0,r)'$ where $\frac{a}{3}+1\leq r<\frac{5a-6}{12}$ have the signature $(-,-,+)$, \label{el:4}
\item $\rho\rho'\alpha(0,r)''$ where $\frac{5(a-3)}{9}\leq r\leq \frac{2a}{3}-1$ have the signature $(-,+,-)$. \label{el:5}    
\end{enumerate} 

From elements in (\ref{el:1}), we can get only squares $\omega^2$ such that $\omega\in \Z$, which are included in the statement of the lemma.

We will proceed with elements in (\ref{el:2}) and (\ref{el:3}), which have the same signature. 
Let $\alpha_1=-(r_1+1)g_2+g_3$ and $\alpha_2=-(r_2+1)g_2+g_3$ where $-1\leq r_1,r_2\leq \frac{a}{3}-1$. Note for $r_i=-1$, we obtain the element $\frac{1+\rho+\rho^2}{3}$ from (\ref{el:2}). Then
\begin{multline*}
(\rho\rho'\alpha_1''+\rho\rho'\alpha_2'')^2=\rho^2\rho'^2(\alpha_1''^2+2\alpha_1''\alpha_2''+\alpha_2''^2)\\>4(a+1)^2\left(1+\frac{1}{a+2}\right)^2\frac{1}{9}\left(1-\frac{1}{a+2}+\frac{1}{(a+3)^2}\right)^2>\gamma''.
\end{multline*} 
Thus, the square of a sum of elements in (\ref{el:2}) and (\ref{el:3}) cannot be totally smaller than $\gamma$.

Let us now focus on elements in (\ref{el:4}). For them, we have
\[
\rho''^2\rho^2\alpha(0,r)'^2>\frac{1}{(a+3)^2}(a+1)^2\left(\frac{a+5}{3}+\frac{a+8}{3(a+2)}+\frac{2}{3(a+2)^2}\right)^2>\frac{1}{2}\gamma
\]
for $\frac{a}{3}+1\leq r<\frac{5a-6}{12}$. That means that the square of a sum of two elements in (\ref{el:4}) cannot be totally smaller than $\gamma$.

Likewise, if we consider elements in (\ref{el:5}), we obtain 
\[
\rho''^2\rho^2\alpha(0,r)'^2>\frac{1}{(a+3)^2}(a+1)^2\left(\frac{5a}{9}-1+\frac{5a}{9(a+2)}+\frac{2}{3(a+2)^2}\right)^2>\frac{1}{2}\gamma''
\] 
for $\frac{5(a-3)}{9}\leq r\leq \frac{2a}{3}-1$. That completes the proof.  
\end{proof}

Now, we are able to prove that the Pythagoras number of $\O_K$ is $6$ as for $\Z[\rho]$.

\begin{pro} \label{prop:pythagoras6}
We have $\P(\O_K)=6$.
\end{pro}

\begin{proof}
If $a<63$, i.e, for $a=21,30,48,57$, we can use a computer program to check that we need at least $6$ squares to express $\gamma$. Mathematica notebook with our code is available at \url{https://sites.google.com/view/tinkovamagdalena/research}. Thus, let us assume that $a\geq 63$. In Lemmas \ref{lem:pythunitexcept}, \ref{lem:pythmono}, \ref{lem:pythline}, \ref{lem:pythtriangle} and \ref{lem:pythdecompo}, we have determined all squares which can be totally smaller than $\gamma$. We have obtained the following elements:
\begin{enumerate}
\item $n^2\in \Z$ where $n^2<a^2$, \label{sq:1}
\item $\rho'^2\rho''^2\left(\frac{1+\rho+\rho^2}{3}\right)^2=\frac{a^2+4a+6}{9}g_1+\frac{a^2+2a+3}{9}g_2-\frac{a}{3}g_3$, \label{sq:2} 
\item $\rho'^2\rho''^2\left(-(r+1)g_2+g_3\right)^2=\frac{a^2+10a+33+(6a+36)r+9r^2}{9}g_1+\frac{a^2+8a+3+6ar}{9}g_2-\frac{a+6+6r}{3}g_3$
\label{sq:3}
where $0\leq r\leq \frac{a}{3}-1$,
\item $\rho''^2\rho^2\left(-g_1'-(r+1)g_2'+2g_3'\right)^2=\frac{4a^2+10a+6-(12a+18)r+9r^2}{9}g_1+\frac{3a^2+2a+3-6ar}{9}g_2-(a-2r)g_3$
where $\frac{a}{3}+1\leq r<\frac{5a-6}{12}$,
\label{sq:4}
\item $\rho^2\rho'^2\left(-g_1''-(r+1)g_2''+2g_3''\right)^2=\frac{15-8a+36r+9r^2}{9}g_1+\frac{3-4a-4a^2+(12a+18)r}{9}g_2+\frac{4a-3-12r}{3}g_3$
where $\frac{5(a-3)}{9}\leq r\leq \frac{2a}{3}-1$.
\label{sq:5}
\end{enumerate}
Recall that only these elements can appear in a decomposition of $\gamma$ to a sum of squares.

First of all, let us note that for all these elements, the coefficients before $g_2$ are nonnegative, and the coefficients before $g_3$ are negative or zero. This is clear for Cases (\ref{sq:1}), (\ref{sq:2}) and (\ref{sq:3}). For (\ref{el:4}), we see that
\[
\frac{3a^2+2a+3-6ar}{9}>\frac{a^2+10a+6}{18}>0
\]
and 
\[
-(a-2r)<-1-\frac{a}{6}<0.
\] 
Similarly, for (\ref{sq:5}), we obtain
\[
\frac{3-4a-4a^2+(12a+18)r}{9}\geq \frac{8a^2-42a-81}{27}>0
\] 
and 
\[
\frac{4a-3-12r}{3}\leq-\frac{8a-51}{9}<0.
\]

Recall that 
\begin{align*}
\gamma&=\frac{34a^2+15a+783}{81}+\frac{11a^2-29a-39}{27}\rho-\frac{11a-33}{27}\rho^2
\\
&=\frac{2(17a^2+24a+342)}{81}g_1+\frac{11a^2-18a-72}{27}g_2-\frac{11(a-3)}{9}g_3.
\end{align*}
If $a$ is even, then $\frac{11(a-3)}{9}$ is odd. If we consider the coefficient before $g_3$ for elements in (\ref{sq:1})--(\ref{sq:5}), we see that this coefficient is odd only for squares in (\ref{sq:5}) under the assumption that $a$ is even. Thus, at least one element from (\ref{sq:5}) must appear in every decomposition of $\gamma$ to a sum of squares. Similarly, let $a$ be odd. Then $\frac{11a^2-18a-72}{27}$, i.e., the coefficient before $g_2$ of $\gamma$, is odd. If we again consider the squares in (\ref{sq:1})--(\ref{sq:5}), only elements in (\ref{sq:5}) has an odd coefficient before $g_2$ for odd $a$. That implies that for all possible $a\geq 63$, at least one element from (\ref{sq:5}) appear in every decomposition of $\gamma$ to a sum of squares.

Considering elements from (\ref{sq:5}), we can conclude that $-\frac{11(a-3)}{9}=\frac{4a-3-12r}{3}$ only if $r=\frac{23a-42}{36}$. For this concrete $r$, the coefficient of an element from (\ref{sq:5}) before $g_2$ is
\[
\frac{3-4a-4a^2+(12a+18)r}{9}=\frac{11a^2}{27}-\frac{13a}{18}-2.
\]
Moreover, 
\[
\frac{11a^2}{27}-\frac{13a}{18}-2=\frac{11a^2-18a-72}{27}
\]
only if $a=12$, which does not give a simplest cubic field from our subfamily. Thus, if the element from (\ref{sq:5}) with $r=\frac{23a-42}{36}$ were in the decomposition of $\gamma$, then some other summand would have to have
\begin{itemize}
\item a non-zero coefficient before $g_2$,
\item a zero coefficient before $g_3$.
\end{itemize} 
However, these conditions are not both satisfied for any element from (\ref{sq:1})--(\ref{sq:5}). Thus, the element from (\ref{sq:5}), which appears in the decomposition of $\gamma$, satisfies
\[
-\frac{11(a-3)}{9}\neq \frac{4a-3-12r}{3}.
\] 
That means that we need at least one more element from (\ref{sq:2})--(\ref{sq:5}) to express $\gamma$.

We know that for all elements from (\ref{sq:5}), the coefficient before $g_3$ satisfies
\[
\frac{4a-3-12r}{3}\leq-\frac{8a-51}{9}.
\]
Thus, the other summand from (\ref{sq:2})--(\ref{sq:5}) has this coefficient between zero and
\[
-\frac{11(a-3)}{9}+\frac{8a-51}{9}=-\frac{a}{3}-2.
\] 
That is true for
\begin{enumerate}[(a)]
\item the element $\frac{a^2+4a+6}{9}g_1+\frac{a^2+2a+3}{9}g_2-\frac{a}{3}g_3$ from (\ref{sq:2}), \label{ssq:1}
\item the element with $r=0$ from (\ref{sq:3}), \label{ssq:2} 
\item elements from (\ref{sq:4}). \label{ssq:3}
\end{enumerate}
Let us now discuss the case when at least two elements from \ref{ssq:1}--\ref{ssq:3} appeared in the decomposition of $\gamma$. In this case, it suffices to realize that a sum of two elements from \ref{ssq:1}--\ref{ssq:3} never has the coefficient before $g_3$ between $-\frac{a}{3}-2$ and $0$, which we have derived to be a necessary condition in this part of the proof. This follows from the fact that the largest coefficient before $g_3$ is achieved by the element in \ref{ssq:3} with the largest $r$, where we have
\[
-(a-2r)<-1-\frac{a}{6}.
\]
So its twice is 
\[
-2(a-2r)<-\frac{a}{3}-2.
\]
It implies that in every decomposition of $\gamma$, there must appear exactly one element from (\ref{sq:5}), and exactly one element from \ref{ssq:1}--\ref{ssq:3}. The remaining summands belong to (\ref{sq:1}).

Let us now focus on elements from \ref{ssq:3}, i.e., elements from (\ref{el:4}), and let us assume that in a decomposition of $\gamma$, there is one element from (\ref{el:4}) and one from (\ref{el:5}). If we consider coefficients before $g_2$ and $g_3$, we can obtain the following system of equations:
\[
\frac{11a^2-18a-72}{27}=\frac{3a^2+2a+3-6ar_1}{9}+\frac{3-4a-4a^2+(12a+18)r_2}{9}
\]
and
\[
-\frac{11(a-3)}{9}=-(a-2r_1)+\frac{4a-3-12r_2}{3}
\] 
where $\frac{a}{3}+1\leq r_1<\frac{5a-6}{12}$ and $\frac{5(a-3)}{9}\leq r_2\leq \frac{2a}{3}-1$. This system has the solution $r_1=\frac{a}{3}-1$ and $r_2=\frac{5(a-3)}{9}$. However, $r_1=\frac{a}{3}-1$ does not belong to our interval. Thus, the elements from \ref{ssq:3} cannot appear in any decomposition of $\gamma$ to a sum of squares.

We will proceed with \ref{ssq:1}. In this case, the coefficient before $g_3$ leads to the equation
\[
-\frac{11(a-3)}{9}=-\frac{a}{3}+\frac{4a-3-12r}{3}
\]
with the solution $r=\frac{10a-21}{18}$. However, regarding the coefficient before $g_2$, we obtain
\[
\frac{a^2+2a+3}{9}+\frac{3-4a-4a^2+(12a+18)r}{9}=\frac{11a^2-18a-45}{27}\neq \frac{11a^2-18a-72}{27}
\]
for every $a$.

Thus, only the element in \ref{ssq:2} can appear in every decomposition of $\gamma$. Moreover, $r=\frac{5(a-3)}{9}$ for the element from (\ref{sq:5}) in this decomposition, and, thus, 
\begin{multline*}
\gamma=\left(\frac{a^2+10a+33}{9}g_1+\frac{a^2+8a+3}{9}g_2-\frac{a+6}{3}g_3\right)\\+\left(\frac{25a^2-42a-180}{81}g_1+\frac{8a^2-42a-81}{27}g_2-\frac{8a-51}{9}g_3\right)+7.
\end{multline*} 
To express $7$, we need at least $4$ squares. Therefore, we need at least $6$ squares to express $\gamma$, which gives $\P(\O_K)= 6$ as desired.    
\end{proof}

\subsection{Universal quadratic forms}

In this part, we will use the results on indecomposables in $\O_K$ to get information on universal quadratic forms over $\O_K$. Before that, we will summarize several basic facts on quadratic forms over number fields.

Let us consider quadratic form of the form $Q(x_1,x_2,\ldots,x_n)=\sum_{1\leq i\leq j\leq n}a_{ij}x_ix_j$ where $a_{ij}\in\O_K$. We say that $Q$ is totally positive definite if $Q(\gamma_1,\gamma_2,\ldots,\gamma_n)\in\O_K^{+}$ for every $n$-tuple $\gamma_1,\ldots,\gamma_n\in\O_K$ such that elements $\gamma_i$ are not all equal to zero. The quadratic form $Q$ is classical if $2$ divides $a_{ij}$ for all $i\neq j$. The quadratic form $Q$ is called diagonal if $a_{ij}=0$ whenever $i\neq j$. And finally, $Q$ is universal if it represents all totally positive algebraic integers in $K$. In this part, we will assume that all quadratic forms are totally positive definite.

To bound the minimal number of variables of universal quadratic forms over our subfamily of the non-monogenic simplest cubic fields, we will use the method developed in \cite[Section 7]{kalatinkova}. This approach is based on the knowledge of indecomposable integers in $K$. In this part, we will use the following notation. Let $S$ denote the set of indecomposable integers in $\O_K$ up to multiplication by squares of units in $\O_K$. Moreover, let us assume that there exist $\delta\in\O_K^{\vee,+}$ and $\alpha_1,\alpha_2,\ldots,\alpha_n\in\O_K^{+}$ such that $\text{Tr}(\alpha_i\delta)=1$ for all $1\leq i\leq n$. Then, a more general result for totally real number fields derived in \cite[Section 7]{kalatinkova} implies the following theorem for cubic fields:

\begin{teo}[{\cite[Section 7]{kalatinkova}}] \label{thm:quaformsbounds}
Let $K$ be a totally real cubic field. 
\begin{enumerate}
\item There exist a diagonal universal quadratic form over $\O_K$ with $\P(\O_K)\#S$ variables.
\item Every classical universal quadratic form over $\O_K$ has at least $\frac{n}{3}$ variables.
\item If $n\geq 240$, then every (non-classical) universal quadratic form over $\O_K$ has at least $\frac{\sqrt{n}}{3}$ variables.    
\end{enumerate}
\end{teo}  

Now we will use Theorem \ref{thm:quaformsbounds} and results on indecomposable integers to derive bounds on the minimal number of variables of universal quadratic forms over our nice subfamily of the simplest cubic fields.

\begin{pro}\label{prop:nonmonogenicbounds}
Let $K$ be a simplest cubic field such that $a\equiv3,21\;(\textup{mod 27})$, $a>12$ and $\frac{\Delta}{27}$ is square-free. Then:
\begin{enumerate}
\item There exists a diagonal universal quadratic form over $\O_K$ with $\frac{a^2+3a}{3}+12a+12$ variables.
\item Every classical universal quadratic form over $\O_K$ has at least $\frac{a^2+3a}{54}$ variables.
\item If $a> 64$, then every (non-classical) universal quadratic form over $\O_K$ has at least $\frac{\sqrt{a^2+3a}}{9\sqrt{2}}$ variables.
\end{enumerate}
\end{pro} 

\begin{proof}
Regarding the first part of the statement, we know that $\P(\O_K)=6$ by Proposition \ref{prop:pythagoras6}. Moreover, in $K$, every totally positive unit is a square. Thus $\#S=\frac{a^2+3a}{18}+2a+2$, which is the number of indecomposable integers up to multiplication by totally positive units.

To get the largest possible value in (2) and (3), it is convenient to consider the triangle of indecomposables $\alpha$ in (v) of Theorem \ref{teolistindec}, for which we have found one totally positive element $\delta$ of the codifferent satisfying $\text{Tr}(\alpha\delta)=1$. In this case, $n=\frac{a^2+3a}{18}$, which gives parts (2) and (3) of the proposition.    
\end{proof}

\section*{Acknowledgements}

The authors want to thank Vít\v{e}zslav Kala and Giacomo Cherubini for enriching discussions and helpful suggestions on the contents of this work. The authors also wish to express their thanks to László Remete and Stéphane Louboutin for their quick answers. The first author was supported by Czech Science Foundation GA\v{C}R, grants 21-00420M and 24-11088O, by Charles University Research Centre program UNCE/SCI/022, and by the \textit{Ministerio de Ciencia e Innovación}, grant PID2022-136944NB-I00. The second author was supported by Czech Science Foundation GA\v{C}R, grant 22-11563O.

\printbibliography[heading=bibintoc]

\end{document}